\documentclass[11pt]{article}
\usepackage{amsmath,amssymb,amsthm,amscd,esint}
\usepackage{url,endnotes,hyperref}

\numberwithin{equation}{section}

\newtheorem{theorem}{Theorem}[section]
\newtheorem{lemma}[theorem]{Lemma}
\newtheorem{definition}[theorem]{Definition}

\theoremstyle{corollary}
\newtheorem{corollary}[theorem]{Corollary}
\theoremstyle{conjecture}
\newtheorem{conjecture}{Conjecture}
\newtheorem{question}[theorem]{Question}

\newtheorem*{claim}{Claim}
\theoremstyle{assumption}

\theoremstyle{proposition}
\newtheorem{proposition}[theorem]{Proposition}
\theoremstyle{remark}
\newtheorem{remark}[theorem]{Remark}
\numberwithin{equation}{section}
\everymath{\displaystyle}

\newcommand{\noop}[1]{}
\newcommand{\CC}{\mathbb{C}}

\newcommand{\RR}{\mathbb{R}}
\newcommand{\ZZ}{\mathbb{Z}}
\newcommand{\TT}{T}

\newcommand{\rec}{\operatorname{rec}}
\newcommand{\Spec}{\operatorname{Spec}}
\newcommand{\Hom}{\operatorname{Hom}}
\newcommand{\Ric}{\operatorname{Ric}}

\DeclareMathOperator{\Hess}{Hess}

\DeclareMathOperator{\Id}{Id}

\begin{document}

\title{\texorpdfstring{\makebox[\textwidth][c]{%
Compactness and Rigidity of Complete K\"ahler Ricci Shrinkers}}%
{Compactness and Rigidity of Complete Kähler-Ricci Shrinkers}}
\author{}
\date{}
\maketitle

% 缩小标题与作者间距 \vspace{-1.2em}，名字放大\Large
\vspace{-3em}
\begin{center}
\Large
\begin{tabular}{cc}
Tongxin Xu$^*$ & \qquad Zhenlei Zhang$^\dagger$
\end{tabular}
\end{center}
% 缩小作者和摘要/正文距离
\vspace{0.5em}

\begin{abstract}
In this paper, we study compactness, rigidity, and related geometric properties of complete K\"ahler--Ricci shrinkers through the polarized Fano fibration structure.
\end{abstract}
\maketitle
\tableofcontents

\section{Introduction}
In geometric flows, one of the most important questions is understanding the formation of singularity. The Ricci flow was introduced by Hamilton as an evolution equation for Riemannian metrics \cite{hamilton1982three}. In his study of self-similar solutions, Hamilton introduced Ricci solitons\cite{MR954419,MR1375255}, which evolve under the Ricci flow only by diffeomorphisms and scaling. Recall that a Ricci shrinker $(X,g,f)$ is a complete Riemannian manifold $(X,g)$ together with a smooth function $f$ satisfying
\begin{align*}
    \operatorname{Ric}+\nabla^2 f=\lambda g, \qquad \lambda>0.
\end{align*}

Ricci shrinkers play a fundamental role in singularity analysis of Ricci flow. In particular, suitable parabolic blow-ups at Type I singular points subconverge to nontrivial gradient shrinking Ricci solitons\cite{naber2007noncompact,MR2886712}. Thus, understanding complete Ricci shrinkers is an essential step toward understanding finite-time singularities of the Ricci flow.

Complete Ricci shrinkers are fully understood in real dimensions two and three (cf. \cite{MR954419,Ivey1993CompactThreeManifolds,MR1375255,Perelman2003RicciFlowSurgery, NiWallach2008Classification,CaoChenZhu2008RecentDevelopments,chen2009strong}, etc.). The complete lists of $\mathbb{R}^2,S^2,$ $\mathbb{R}^3,S^3,S^2\times\mathbb{R}$, as well as their corresponding quoients. In real dimension four, however, a complete classification remains out of reach. The geometry is substantially richer: there exist non-Einstein compact shrinkers, such as the Koiso--Cao solitons, as well as noncompact examples including the Feldman--Ilmanen--Knopf shrinker and the more recently constructed BCCD shrinker \cite{koiso1990rotationally,cao1996existence,feldman2003rotationally,bamler2024new}. Nevertheless, substantial progress has been made on four-dimensional shrinkers (cf.\cite{naber2007noncompact,CaoChen2013BachFlat,haslhofer2015note,munteanu2015geometry,kotschwar2015rigidity,LiNiWang2018PIC,MunteanuWang2019Infinity},etc.) For comprehensive treatments of these and related topics, we refer the reader to the monographs \cite{chow2023ricci,chow2025ricci}.

A particularly important class in real even dimensions is formed by Kähler Ricci shrinkers. A Ricci shrinker $(X,g,f)$ is called a Kähler Ricci shrinker if it admits a Kähler structure $J$ such that $\nabla f$ is a holomorphic vector field. In this case, one can rewrite the shrinker equation as
\begin{align*}
    \operatorname{Ric(\omega)}+\sqrt{-1}\partial\bar{\partial} f= \lambda \omega, \qquad \lambda>0,
\end{align*}
where $\omega$ is the Kähler form induced by $g$ and $J$, and $\operatorname{Ric}(\omega)$ its Ricci form. The additional complex structure brings methods from complex and algebraic geometry into the study of Ricci-flow singularities.

Building on the foundational work by Munteanu‑Wang \cite{MunteanuWang2019Infinity} and Cifarelli \cite{cifarelli2022uniqueness}. Complete Kähler Ricci shrinker surfaces have recently been fully classified (cf. \cite{conlon2024classification,cifarelli2024finite,bamler2024new,li2026kahler}). The noncompact models include the Gaussian shrinker over $\mathbb{C}^2$, the standard product shrinker over $\mathbb P^1\times\mathbb C$, the FIK shrinker over $\operatorname{Tot}\mathcal O_{\mathbb P^{1}}(-1)$, and the BCCD shrinker over $\operatorname{Bl}_{p}(\mathbb P^1\times\mathbb C)$. More recently, Conlon, Hallgren, and Ma \cite{conlon2025non} gave an alternative proof of the curvature boundedness theorem of  Li and Wang. They also proved that every volume-noncollapsed finite-time singularity of the Kähler Ricci flow on a compact Kähler surface is of Type I. Combining this characterization with \cite{cifarelli2024finite}, one concludes that every volume-noncollapsed finite-time singularity of the Kähler Ricci flow on a compact Kähler surface is modeled on the FIK shrinker.

The preceding classification results reveal the geometry of complete
K\"ahler--Ricci shrinkers in complex dimension two. In higher
dimensions, Sun--Zhang \cite{sun2024k} established a new
algebro-geometric framework by proving that every smooth complete
K\"ahler--Ricci shrinker is a quasi-projective variety admitting a
canonical polarized Fano fibration. This perspective has been extended to singular shrinkers. Hallgren
\cite{hallgren2023k} proved that tangent cones of singular shrinkers
arising from noncollapsed K\"ahler--Ricci flows are normal affine
algebraic varieties, and Hallgren--Zhang \cite{hallgren2026singular}
showed that the singular shrinkers themselves carry locally algebraic
structures with log terminal singularities. Further developments by Li--Zhang \cite{LiZhang26} related the
polarized Fano fibration to the asymptotic geometry of shrinkers,
showing that the fibration determines the local cone structure at
infinity and characterizing asymptotic conicality. On the rigidity
side, Conlon--Deruelle \cite{ConlonDeruelle26} proved that shrinkers on
resolutions of K\"ahler cones are asymptotically conical and, combined
with Esparza's uniqueness theorem \cite{esparza2025uniqueness},
obtained uniqueness of such shrinkers up to biholomorphism.

In this paper, we study smooth complete K\"ahler--Ricci shrinkers.
By the work of Sun--Zhang \cite{sun2024k}, every such shrinker admits
a canonical surjective projective morphism
\[
    \pi:X\longrightarrow Y=\operatorname{Spec}R_X
\]
onto a polarized affine cone. The morphism is equivariant with
respect to the compact torus \(T\) obtained as the closure of the
one-parameter group generated by \(J\nabla f\). The induced \(T\)-action
gives a weight decomposition
\[
    R_X=\bigoplus_{\beta\in\operatorname{Lie}(T)^*}R_{X,\beta}.
\]
For a homogeneous function \(\phi\in R_{X,\beta}\), we have
\[
    \nabla f(\phi)=\alpha_\beta\phi,
    \qquad
    -\Delta_f\phi=\alpha_\beta\phi,
\]
where $\alpha_\beta=\langle\beta,\xi\rangle>0$, and \(\xi\in\operatorname{Lie}(T)\) corresponds to \(J\nabla f\).
We denote by $E_1=\{\phi\in R_X:\nabla f(\phi)=\phi\}$ the weight-one subspace, which plays a fundamental role in the
splitting analysis. To extract geometric information from homogeneous regular functions,
we introduce the following notion.
\begin{definition}[First-order visibility]
A K\"ahler--Ricci shrinker \((X,g,J,f)\) is called
first-order visible if there exist a point \(q\in X\) and a
nonconstant regular function \(\phi\in R_X\) such that
\[
    \nabla f(q)=0,\qquad d\phi(q)\neq0 .
\]
\end{definition}

First-order visibility provides a simple mechanism for extracting local
geometric information from the global algebraic structure of a
K\"ahler--Ricci shrinker. Indeed, after decomposing \(\phi\) into
homogeneous components, one obtains a homogeneous regular function
whose differential is nonzero at a fixed point of the soliton torus.
Its weight determines an eigenvalue of the drifted Laplacian, while
linearizing the soliton vector field at the fixed point relates the
same weight to an eigenvalue of the Ricci tensor. Thus first-order
visibility gives a direct bridge between the polarized Fano fibration,
weighted spectral theory, and local Ricci geometry. This viewpoint
underlies the rigidity, splitting, and curvature-signature results
developed below.

The spectral theory of the drifted Laplacian on complete Ricci
shrinkers is well developed. Under a Bakry--\'Emery Ricci curvature
lower bound, the weighted Lichnerowicz estimate gives a sharp lower
bound for the first nonzero eigenvalue. Cheng--Zhou
\cite{cheng2017eigenvalues} characterized the equality case by a
Gaussian splitting $\mathbb R^k\times N$, where \(k\) is determined by the multiplicity of the first eigenvalue.
In the K\"ahler setting, He--Ou \cite{he2024dimension} further studied
the relation between polynomial-growth holomorphic functions and the
drifted Laplacian spectrum.

Our starting point is to apply this weighted spectral theory to the algebraic weight decomposition arising from the polarized Fano fibration. A useful analytic input is the automatic weighted \(L^2\)-integrability of homogeneous functions. This allows the sharp weighted Lichnerowicz estimate and its equality case to be applied directly to the coordinate ring. We record the resulting K\"ahler formulation below and include a self-contained proof for later use.

\begin{theorem}[Spectral Gap and Splitting]
Let $(X^n,g,J,f)$ be a complete K\"ahler-Ricci shrinker.
Then for every nonconstant homogeneous regualr function  $\phi\in R_{X,\beta}$, we have $\alpha_{\beta} \ge 1$. More precisely:
\begin{enumerate}
\item If $\operatorname{Ric} > 0$ everywhere on $X$, then $\alpha_{\beta} > 1$.
\item If $\alpha_{\beta} = 1$, then $X$ is holomorphically isometric to $\mathbb{C} \times N$, where $N$ itself is a complete K\"ahler-Ricci shrinker.
\item If $\dim_{\mathbb{C}}E_1=k$, then $X$ is holomorphically isometric to $\mathbb{C}^k \times N$, where $N$ itself is a complete K\"ahler-Ricci shrinker.
\end{enumerate}
\end{theorem}

We next combine this spectral information with the first-order geometry of the polarized Fano fibration to derive several rigidity results.

The first geometric application is the rigidity and splitting of Kähler Ricci shrinker. A basic rigidity problem asks to what extent positivity of the Ricci tensor constrains the global geometry of a complete Ricci shrinker. In particular, the following question was recently emphasized in \cite{chow2025ricci}.

\begin{question}
Must every complete Ricci shrinker satisfying \(\operatorname{Ric}>0\) be compact?
\end{question}

This question has been answered affirmatively under several stronger curvature assumptions. In the Riemannian setting, Perelman showed that every three-dimensional noncollapsed shrinker with positive bounded sectional curvature is compact \cite{Perelman2003RicciFlowSurgery}. In arbitrary dimensions, Munteanu and Wang \cite{munteanu2017positively} proved that every complete Ricci shrinker with nonnegative sectional curvature and positive Ricci curvature is compact. Li and Ni \cite{LiNi2020Orthogonal} subsequently established compactness under the weaker assumption of weakly PIC$_1$ curvature together with positive Ricci curvature. More recently, Wu and Wu \cite{wu2023shrinkers} obtained compactness under positive second-Ricci curvature, as well as under several related curvature-pinching conditions. Further related results can be found in \cite{NiWallach2008Classification,PetersenWylie2010Classification,naber2007noncompact,CaoChen2013BachFlat,LiNiWang2018PIC,QuWu2024Compact}, among others.

In the K\"ahler setting, Ni \cite{ni2005ancient} proved compactness under nonnegative holomorphic bisectional curvature and positive Ricci curvature; an alternative proof was later given by Wu and Zhang \cite{wu2016remarks}. Zhang \cite{zhang2019gradient}, and independently Li and Ni \cite{LiNi2020Orthogonal}, extended this result to nonnegative orthogonal bisectional curvature. In complex dimension two, Munteanu and Wang \cite{MunteanuWang2019Infinity} proved compactness assuming positive Ricci curvature and bounded curvature. Since Li and Wang \cite{li2026kahler} subsequently showed that every K\"ahler--Ricci shrinker surface has bounded sectional curvature, it follows that positive Ricci curvature alone implies compactness in complex dimension two.

In this paper, we give a new spectral--algebraic proof of these
compactness and rigidity phenomena for K\"ahler--Ricci shrinker
surfaces, and extend the same mechanism to the toric setting.

\begin{theorem}\label{main}
Let \((X^2,g,J,f)\) be a complete K\"ahler--Ricci shrinker surface.
\begin{enumerate}
    \item If \(\operatorname{Ric}>0\), then \(X\) is compact, hence a
    Fano surface;
    \item If \(X\) is noncompact and \(\operatorname{Ric}\geq 0\), then
    \(X\) is holomorphically isometric to either
    \(\mathbb C^2\) or \(\mathbb P^1\times\mathbb C\), equipped with
    the Gaussian or product shrinker, respectively.
\end{enumerate}
\end{theorem}

Although complete K\"ahler--Ricci shrinker surfaces have recently been
classified, the classification alone does not reveal the local Ricci
geometry of the resulting models. Our approach is independent of the
surface classification and uses neither curvature bounds,
noncollapsing assumptions, nor asymptotic analysis. It also yields
local curvature information that is not apparent from the explicit
construction of the known examples.

In the toric setting, first-order visibility is automatic for
noncompact shrinkers. Indeed, an unbounded edge of the moment
polyhedron produces a global homogeneous toric character whose
differential is nonzero at the fixed point corresponding to one of its
vertices. Thus the combinatorics of the moment polyhedron supplies the
required first-order information. Combined with the weighted spectral
gap, this gives the following.

\begin{theorem}
Let \((X^n,g,J,f)\) be a complete toric K\"ahler--Ricci shrinker.
\begin{enumerate}
    \item If \(\operatorname{Ric}>0\), then \(X\) is compact, hence a
    toric Fano manifold;
    \item If \(X\) is noncompact and \(\operatorname{Ric}\geq 0\), then
    \(X\) is \(T_{\mathbb C}\)-equivariantly holomorphically isometric
    to \(\mathbb C\times N\), where \(N\) is itself a complete toric
    K\"ahler--Ricci shrinker.
\end{enumerate}
\end{theorem}

The same first-order mechanism also detects mixed Ricci signature in
explicit noncompact shrinkers. For the BCCD shrinker, we obtain the
following.

\begin{corollary}\label{BCCD}
Let \((X,g,J,f)\) be the BCCD shrinker $\operatorname{Bl}_{p}(\mathbb P^1\times\mathbb C)$, and let \(C\cup E\) be the reducible singular fiber of the natural
fibration $\operatorname{Bl}_{p}(\mathbb P^1\times\mathbb C)\longrightarrow
\mathbb C$. Then there exist points
\[
q_C\in C\setminus E,
\qquad
q_E\in E\setminus C,
\]
such that the Ricci tensor has a strictly negative eigenvalue at both
\(q_C\) and \(q_E\). Consequently, the Ricci tensor has mixed signature
at these two points.
\end{corollary}

This conclusion does not rely on an explicit formula for the BCCD
metric; it follows instead from the first-order behavior of the
polarized Fano fibration together with the weight structure of its
coordinate ring.

The above argument illustrates that first-order visibility can detect
local Ricci curvature information. The same mechanism also leads to
global rigidity results when the underlying complex manifold admits
sufficiently many global regular functions.

\begin{conjecture}\cite[Conjecture 6.3]{sun2024k}
On a Fano cone \(X\), every K\"ahler--Ricci shrinker is a Ricci-flat
K\"ahler cone metric.
\end{conjecture}

We verify this conjecture when the underlying affine variety is smooth,
including the cone vertex. Our result holds in arbitrary complex
dimension and requires no a priori curvature bounds, noncollapsing
assumptions, or asymptotic hypotheses.

\begin{theorem}[Smooth Fano Cone Rigidity]\label{SZCX}
Every K\"ahler--Ricci shrinker on a smooth Fano cone is
holomorphically isometric to the Gaussian shrinker on
\(\mathbb C^n\).
\end{theorem}

More generally, the same rigidity mechanism applies whenever the
underlying complex manifold is Stein, yielding the following
intrinsic formulation.

\begin{theorem}[Stein Rigidity]\label{GGG}
Let \((X^n,g,J,f)\) be a complete K\"ahler--Ricci shrinker. If the
underlying complex manifold \(X\) is Stein, then
\((X,g,J,f)\) is holomorphically isometric to the Gaussian shrinker
on \(\mathbb C^n\).
\end{theorem}

While this manuscript was being completed, we became aware of the
independent work of Conlon--Deruelle
\cite{ConlonDeruelle26}. The conclusions of
Theorems~\ref{SZCX} and~\ref{GGG} also follow from their results.
Indeed, a Stein manifold is \(1\)-convex with empty exceptional set,
and hence \cite[Proposition~2.9]{ConlonDeruelle26} implies that the
underlying complex manifold of such a shrinker is biholomorphic to
\(\mathbb C^n\). Viewing \(\mathbb C^n\) as the trivial resolution of
the flat K\"ahler cone, their
\cite[Corollary~B]{ConlonDeruelle26} then implies that the shrinker is,
up to pullback by a biholomorphism, the Gaussian shrinker. Since every
smooth affine cone is Stein, this also yields the smooth Fano cone
case. Their approach and ours are substantially different.
Conlon--Deruelle establish asymptotic conicality for shrinkers on
resolutions of K\"ahler cones and then invoke uniqueness results for
asymptotically conical shrinkers. By contrast, our proof is local and
spectral-algebraic: it combines the weighted spectral gap with the
first-order behavior of homogeneous regular functions at a fixed
point, and derives Gaussian splitting without first establishing
curvature decay or asymptotic conicality.

This paper is organized as follows. Section~2 reviews basic
background on Ricci shrinkers, toric geometry, polarized Fano
fibrations, and algebraic preliminaries. In Section~3, we combine the
polarized Fano fibration structure with the weighted spectral theory
of homogeneous regular functions and establish the spectral gap
theorem with its rigidity consequences. Section~4 studies
first-order visibility for noncompact K\"ahler--Ricci shrinkers and
applies it to establish the main rigidity results for
surfaces and toric shrinkers. Section~5 applies first-order
visibility to analyze the Ricci curvature of noncompact shrinker
examples. In Section~6, we apply this framework to
the Fano cone rigidity problem.

%\subsectionnotoc{Notation and Conventions} 
 
%\subsectionnotoc{Acknowledgments} 
 
%%%%%%%%%%%%%%%%%%%%%%%%%%%%%%%%%%
%%%%%%%%%%%%%%%%%%%%%%%%%%%%%%%%%%

%%%%%%%%%%%%%%%%%%%%%%%%%%%
%%%%%%%%%%%%%%%%%%%%%%%%%%%%%

\section{Preliminaries}

\subsection{Ricci shrinker}
\begin{definition}\label{soliton}
    A Ricci shrinker is a triple $(X,g,f)$, where $X$ is a complete Riemannian manifold with a vector field $\nabla^gf$ satisfying the equation
\[
\Ric(g) + \Hess_g(f) = g,
\]
the vector field $\nabla^g f$ is called the soliton vector field.

If $g$ is complete and K\"ahler with K\"ahler structure $J$, then we say that $(X,g,J,f)$ is a K\"ahler Ricci shrinker, $\nabla^g f$ is complete and real holomorphic, and
\[
\operatorname{Ric}(\omega) + \sqrt{-1}\partial \bar{\partial} f = \omega,
\]
where $\operatorname{Ric}(\omega) $ is the Ricci form of Kähler form $\omega$.
\end{definition}

The above definition implies the standard normalization conditions and soliton identities associated to the shrinker.

\begin{lemma}\label{iden}
    Let $(X, g, f)$ be a n-dimensional Ricci shrinker, with soliton vector field $\nabla^g f$ for a smooth real-valued function $f:M\to \mathbb{R}$. Then we have the following soliton identities:
    \begin{align*}
        R_g+|\nabla^g f|^2=2f, \qquad R_g+\Delta f=n, \qquad R\geq 0.
    \end{align*}
\end{lemma}
\begin{proof}
    The first two inequalities hold by direct computation,  See \cite{chow2023ricci} for the complete proof. $R\geq 0$ follows from \cite{chen2009strong}.
\end{proof}
All subsequent reasoning adopts the conventions introduced in Definition \ref{soliton} and Lemma \ref{iden}, unless extra notation is explicitly indicated.

Another useful result is that the soliton potential of a complete non-compact Ricci shrinker grows quadratically in distance. In particular, $f$ is proper.

\begin{theorem}\cite{cao2010complete}
    Let $(X, g, f)$ be a complete non-compact Ricci shrinker with smooth real-valued potential function $f \colon X \to \mathbb{R}$. Then for all $x \in X$, $f$ satisfies the estimates
\[
\frac14 \bigl(d_g(p,x) - c_1\bigr)^2 - C \le f(x) \le \frac14 \bigl(d_g(p,x) + c_2\bigr)^2
\]
for some constant $C>0$, where $d_g(p,\,\cdot\,)$ denotes the distance to a fixed point $p\in X$ with respect to $g$. Here, $c_1$ and $c_2$ are positive constants depending only on the real dimension of $X$ and the geometry of $g$ on the unit ball $B_p(1)$ centered at $p$.
\end{theorem}
In \cite{cao2010complete}, their normalization is $R+|\nabla f|^2=f$, under the normalization $R+|\nabla f|^2=2f$, the potential function $f$ also has quadratic growth and is proper.

\subsection{Toric geometry}
In this subsection, we revise some standard facts from toric geometry, based on \cite{cifarelli2024finite} and \cite{cox2011graduate}.

We first review some fundamental notion of polyhedrons. Let $E$ be a real vector space of dimension $n$ and let $E^*$ denote the dual. Write \(\langle\cdot,\cdot\rangle\) for the natural pairing \(E^*\times E\to\mathbb{R}\). Let \(\Gamma\subset E\) be a lattice, with dual lattice \(\Gamma^*\subset E^*\).
\begin{definition}\cite[Definition 2.8]{cifarelli2024finite}
    A polyhedron $P$ in $E$ is a finite intersection of closed half spaces:
    \begin{align*}
        P = \bigl\{ m \in E \bigm| \langle m, u_i \rangle \geq -a_i,\ i=1,\dots,s \bigr\}, \qquad u_i \in E^*,\ a_i\in \mathbb{R}.
    \end{align*}
    It is called a polyhedral cone if all $a_i=0$ and set $C$, the dual of polyhedral cone $C$ is defined by $C^\vee = \{ x \in E^* \mid \langle x, C \rangle \ge 0 \}$.
\end{definition}

\begin{definition}\cite[Definition 2.9]{cifarelli2024finite}
    A polyhedron $P \subset E^*$ is called Delzant if its set of vertices is non-empty and each vertex $v \in P$ has the property that there are precisely $n$ edges $\{e_1, \dots, e_n\}$ (one-dimensional faces) emanating from $v$, and there exists a basis $\{\varepsilon_1, \dots, \varepsilon_n\}$ of $\Gamma^*$ such that $\varepsilon_i$ lies along the ray $\mathbb{R}(e_i - v)$.
\end{definition}
The Delzant polyhedron is strongly convex i.e. it does not contain any affine subsapce of $E$. Another important concept is the recession cone(asymptotic cone).
\begin{definition}\cite[Section 7.1]{cox2011graduate}
    Let $P\subseteq E$ be a polyhedron, its recession cone is defined by 
    \begin{align*}
        \operatorname{rec}(P) = \{ m \in E \mid \langle m, u_i \rangle \geq 0,\ i=1,\dots,s \}
    \end{align*}
\end{definition}

An equivalent definition of $\operatorname{rec}(P)$ is the set of $m\in E$ with the property that there exists $m_0\in E$ such that $m+tm_0\in P$ for sufficiently large $t>0$.

\begin{definition}\cite[Definition 2.16]{cifarelli2024finite}\label{toricc}
  A \textit{toric manifold} is an $n$-dimensional complex manifold $X$ endowed with an effective holomorphic action of the algebraic torus $(\mathbb{C}^*)^n$ such that the following hold true:
  \begin{enumerate}
    \item The fixed point set of the $(\mathbb{C}^*)^n$-action is compact.
    \item There exists a point $p \in X$ with the property that the orbit $(\mathbb{C}^*)^n \cdot p \subset X$ forms a dense open subset of $X$.
  \end{enumerate}
\end{definition}
Let $T^n$ denote the compact real form of algebraic torus $(\mathbb{C}^*)^n$ with Lie algebra $\mathfrak{t}$. There exists a natural integer lattice $N_\mathbb{Z}:=\Gamma\cong \mathbb{Z}^n$ in $\mathfrak{t}$, and we set its dual lattice $M_{\mathbb{R}}\cong \mathbb{Z}^n$ in $\mathfrak{t}^*$. The following combinatorial definition plays an important role in toric geometry and can be used to construct toric varieties.
\begin{definition}\cite[Definition 2.17]{cifarelli2024finite}\label{fann}
    A fan $\Sigma$ in $\mathfrak{t}$ is a finite set of rational polyhedral cones $\sigma$ satisfying
\begin{enumerate}
    \item For every $\sigma \in \Sigma$, each face of $\sigma$ also lies in $\Sigma$.
    \item For every pair $\sigma_1, \sigma_2 \in \Sigma$, $\sigma_1 \cap \sigma_2$ is a face of each.
\end{enumerate}
\end{definition}
A toric manifold can always be constructed from a fan.
\begin{proposition}\label{fan}\cite[Corollary 3.18]{cox2011graduate}
    Let $X$ be a smooth toric manifold. Then there exists a fan $\Sigma$ such that $X\cong X_{\Sigma}$.
\end{proposition}

\begin{proposition}[orbit-cone correspondence]\cite[Theorem 3.2.6]{cox2011graduate}
    The $k$-dimensional cones $\sigma \in \Sigma$ are in natural one-to-one correspondence with the $(n - k)$-dimensional orbits $O_\sigma$ of the $(\mathbb{C}^*)^n$-action on $X_\Sigma$. 
\end{proposition}

In this paper, the toric Kähler Ricci shrinker is defined as follows.
\begin{definition}\cite[Definition 2.21]{cifarelli2024finite}
  A complex $n$-dimensional Kähler Ricci shrinker $(X,g,J,f)$ is \textit{toric} if $X$ is a toric manifold, $J\nabla f$ lies in the Lie algebra $\mathfrak{t}$ of the underlying real torus $T^n$ that acts on $X$, and $g$ is $T^n$-invariant. In particular, the zero set of $\nabla f$ is compact.
\end{definition}

\subsection{Polarized Fano fibration}
Here we recall the polarized Fano fibration
framework of \cite{sun2024k}.

A \textit{fibration} is a surjective projective morphism $\pi \colon X \to Y $
between normal varieties such that $\pi_*\mathcal{O}_X = \mathcal{O}_Y$. The following definition is introduced by \cite{collins2019sasaki} to study the canonical metric in Sasaki geometry.
\begin{definition}\cite{collins2019sasaki}
    A \textit{polarized affine cone} $(Y,\xi,T)$ consists of a normal affine variety $Y = \operatorname{Spec} R$ with a compact torus $T$-action admitting a unique fixed point, together with a vector $\xi \in \operatorname{Lie}(T)$ lying in the Reeb cone. Here the latter condition means that for the weight decomposition
\[
R = \bigoplus_{\alpha \in \operatorname{Lie}(T)^*} R_\alpha,
\]
one has $\langle \alpha, \xi \rangle > 0$ whenever $R_\alpha \neq 0$ and $\alpha \neq 0$.
\end{definition}
\begin{definition}[Polarized Fano fibration]
    A polarized Fano fibration $(\pi\colon X \to Y, \xi)$ is a fibration $\pi\colon X \to Y$ with the following properties:
\begin{enumerate}
    \item $\pi\colon X \to Y$ is a Fano fibration. That is, $X$ and $Y$ are normal varieties, $X$ is klt, and that $-K_X$ is $\pi$-ample and $\mathbb{Q}$-Cartier.
    \item $X$ and $Y$ are equipped with a $\pi$-equivariant torus action $T$, and $\xi \in \mathfrak{t} = \operatorname{Lie}(T)$.
    \item $(Y, T, \xi)$ is a polarized affine cone.
\end{enumerate}
\end{definition}
Sun-Zhang proved the following structural theorem via deep algebraic geometric techniques.
\begin{theorem}\cite[Theorem 3.1]{sun2024k}
    A K\"ahler-Ricci shrinker $(X,g,J,f)$ naturally defines a polarized Fano fibration. In particular, $X$ is quasi-projective.
\end{theorem}

\subsection{Algebraic Preliminaries}

In this section, we presents some well-known algebraic results; all proofs are included to keep the exposition self-contained.
\begin{proposition}\label{A1}
    The one-dimensional polarized affine cone $(Y,\xi, T)$ is isomorphic to $\mathbb{A}^1$.
\end{proposition}
\begin{proof}
    Since $Y$ is a one-dimensional normal affine variety, it is necessarily a smooth affine curve. The non-trivial action of the complexified torus $T_{\mathbb{C}}$ on $Y$ guarantees the existence of a dense open orbit $O \subset Y$. Any such one-dimensional orbit is isomorphic to $T_{\mathbb{C}}$ modulo a finite cyclic group, which yields $O \cong \mathbb{C}^* / \mu_m \cong \mathbb{C}^*$. By the definition of a polarized affine cone, $Y$ contains a unique $T$-fixed point $o$. Therefore, $Y$ admits the orbit decomposition:$ Y = O \sqcup {o} $. Topologically and algebraically, $Y$ is a smooth affine curve obtained by adjoining a single point to $\mathbb{C}^*$. The unique smooth projective completion of $O \cong \mathbb{C}^*$ is $\mathbb{P}^1$. Since $Y$ is affine and strictly contains $O$, it must arise from removing exactly one point from $\mathbb{P}^1$. Consequently, we obtain the isomorphism:$ Y \cong \mathbb{P}^1 \setminus {\infty} \cong \mathbb{A}^1 $.
\end{proof}

\begin{proposition}\label{FFF}
    Any compact torus \(T\) acting holomorphically on \(\mathbb{P}^n\) has at least $n+1$ fixed points.
\end{proposition}
\begin{proof}
    Any compact torus holomorphic action induces a Lie group homomorphism
\[
\rho\colon T \to \operatorname{Aut}(\mathbb{P}^n) = \operatorname{PGL}(n+1,\mathbb{C}).
\]
We set $K = \rho(T)$; then $K$ is a compact connected abelian subgroup of $\operatorname{PGL}(n+1,\mathbb{C})$.
Recall that the maximal compact subgroup of $\operatorname{PGL}(n+1,\mathbb{C})$ is $\operatorname{PU}(n+1)$.

Any compact subgroup of $\operatorname{PGL}(n+1,\mathbb{C})$ can be conjugated into $\operatorname{PU}(n+1)$, i.e., there exists $g\in\operatorname{PGL}(n+1,\mathbb{C})$ such that $gKg^{-1}\subseteq\operatorname{PU}(n+1)$.
Moreover, every compact abelian subgroup of $\operatorname{PU}(n+1)$ can be diagonalized: there exists $h\in\operatorname{PU}(n+1)$ satisfying $h(gKg^{-1})h^{-1} \subseteq T_{\mathrm{std}}$,
where $T_{\mathrm{std}} \cong (S^1)^n$ denotes the standard diagonal torus.

Let $\Phi = hg$. Then $\Phi K \Phi^{-1} \subseteq T_{\mathrm{std}}$, and $\Phi$ defines a holomorphic automorphism mapping the fixed-point set of $T$ to the fixed-point set of $\Phi T \Phi^{-1}$.
Since $\Phi T \Phi^{-1} \subseteq T_{\mathrm{std}}$, we have
\[
\operatorname{Fix}(T_{\mathrm{std}}) \subseteq \operatorname{Fix}\big(\Phi T \Phi^{-1}\big).
\]

Note that each homogeneous coordinate point $e_i = [0:0:\cdots:1:\cdots:0]$
is a fixed point for weights $\lambda=(\lambda_1,\dots,\lambda_n)$ with $|\lambda_i|=1$ under the torus action:
\[
\lambda\cdot e_i = [\lambda_1\cdot 0,\ \dots,\ \lambda_i\cdot 1,\ \dots,\ \lambda_n\cdot 0] = e_i.
\]
Hence the action admits at least $(n+1)$ distinct fixed points.
\end{proof}

\section{Spectral Gaps and Gaussian Splittings}
\subsection{Weighted integrability and the spectral gap}

Let \((X,g,J,f)\) be a K\"ahler Ricci shrinker, and let \(T\) be the
compact real torus obtained as the closure of the one-parameter group
generated by \(J\nabla f\). We identify $\operatorname{Lie}(T)$ with the corresponding space of fundamental
vector fields on \(X\), and write $\xi:=J\nabla f\in\operatorname{Lie}(T)$.

Following Sun--Zhang, the pullback action of \(T\) on the ring \(R_X\)
of regular functions gives a weight decomposition
\[
    R_X=\bigoplus_{\beta\in \operatorname{Lie}(T)^*}R_{X,\beta},
\]
where
\[
    R_{X,\beta}
    :=
    \left\{
        \phi\in R_X:
        \mathcal L_{\eta}\phi
        =
        \sqrt{-1}\,\langle\beta,\eta\rangle\phi
        \quad\text{for all }\eta\in \operatorname{Lie}(T)
    \right\}.
\]
The weights that actually occur lie in the character lattice
 $M:=\Hom(T,S^1)\subset\operatorname{Lie}(T)^*$. In particular, \(R_{X,\beta}=0\) for \(\beta\notin M\). Moreover, every nonzero occurring weight satisfies $\langle\beta,\xi\rangle>0$. Let \(0\neq\phi\in R_{X,\beta}\) with \(\beta\neq0\), and set $\alpha_\beta:=\langle\beta,\xi\rangle>0$.
Then
\[
    \mathcal L_{\xi}\phi
    =
    \sqrt{-1}\,\alpha_\beta\phi.
\]
Since \(\phi\) is holomorphic and $\xi$ is a real vector field, we have
\[
    \mathcal L_{\xi}\phi
    =
    d\phi(J\nabla f)
    =
    \sqrt{-1}\,d\phi(\nabla f)
    =
    \sqrt{-1}\,\nabla f(\phi).
\]
Consequently,
\[
    \nabla f(\phi)=\alpha_\beta\phi.
\]
By the holomorphicity of $\phi$, we have
\begin{align*}
    -\Delta_f\phi
    =
    \alpha_\beta\phi,
\end{align*}
where $\Delta_f=\Delta-\nabla f$.

We will use the following weighted integrability lemma for homogeneous
regular functions. This is a special case of the stronger weighted
integrability result of Li--Zhang \cite[Lemma~4.1]{LiZhang26}; we state
it in the form adapted to our spectral argument.

\begin{lemma}[$L^2$ Integrability]\label{3.1}
    Let $(X^m, g, f)$ be a complete Ricci shrinker. Suppose $\phi \in C^\infty(X, \mathbb{C})$ satisfies $\nabla f(\phi) = \alpha \phi $ for any $\alpha > 0$.
Then
\[
\phi \in L^2(e^{-f}dV_g).
\]
\end{lemma}
\begin{proof}
  See \cite[Lemma~4.1]{LiZhang26}.
\end{proof}

$L^2$-integrability yields the following integral Bochner formula.

\begin{proposition}[Global Spectral Lower Bound]\label{3.2}
    Let \((X^m,g,f)\) be a complete Ricci shrinker. Suppose \(u\in C^\infty(X)\cap L^2(e^{-f}dV)\) be a nonconstant function satisfying
\[
-\Delta_f u = \alpha u.
\]
Then 
\[
\int_{X}|\nabla^{2}u|^{2}\,d\mu = (\alpha-1)\int_{X}|\nabla u|^{2}\,d\mu,
\]
which immediately yields \(\alpha \geq 1\). More precisely:
\begin{enumerate}
    \item If \(\operatorname{Ric} > 0\) everywhere on \(X\), then \(\alpha > 1\).
    \item If \(\alpha=1\) holds, then \(\nabla^2 u\equiv 0\) and the universal cover \(\widetilde{X}\) admits an isometric splitting \(\mathbb{R}\times N\), with \(N\) a complete Ricci shrinker.
\end{enumerate}
\end{proposition}

\begin{proof}
    Firstly, we prove $\nabla u \in L^{2}(d\mu)$. Let $d\mu = e^{-f}\, dV_g$. We introduce a cutoff function $\chi_{R}$ satisfying
\[
0 \leq \chi_{R}\leq 1, \quad \chi_{R} \equiv 1 \text{ on } B_{R},\quad \operatorname{supp}\chi_{R} \subset B_{2R},\quad |\nabla\chi_{R}| \leq \frac{C}{R}.
\]
Multiplying both sides of the equation $-\Delta_{f}u = \alpha u$ by $\chi_{R}^{2}u$ and integrating by parts, we compute
\begin{align*}
\int_{X}\chi_{R}^{2}|\nabla u|^{2}\,d\mu
&= \alpha \int_{X}\chi_{R}^{2}u^{2}\,d\mu - 2\int_{X}\chi_{R}u\langle\nabla\chi_{R},\nabla u\rangle\,d\mu \\
&\leq \alpha \int_{X}\chi_{R}^{2}u^{2}\,d\mu + 2\left|\int_{X}\chi_{R}u\langle\nabla\chi_{R},\nabla u\rangle\,d\mu\right| \\
&\leq \alpha \int_{X}\chi_{R}^{2}u^{2}\,d\mu + \frac12\int_{X}\chi_{R}^{2}|\nabla u|^{2}\,d\mu + 2\int_{X}u^{2}|\nabla\chi_{R}|^{2}\,d\mu.
\end{align*}
Rearranging terms gives
\[
\int_{X}\chi_{R}^{2}|\nabla u|^{2}\,d\mu \leq 2\alpha\int_{X}u^{2}\,d\mu + 4\int_{X}u^{2}|\nabla\chi_{R}|^{2}\,d\mu.
\]
Letting $R\to\infty$ and using $u\in L^{2}(d\mu)$, we conclude $\nabla u\in L^{2}(d\mu)$.

Secondly, we prove $\nabla^{2}u \in L^{2}(d\mu)$.
We use the following  Bakry–Émery Bochner formula:
\[
\frac12\Delta_{f}|\nabla u|^{2} = |\nabla^{2}u|^{2} + \operatorname{Ric}_{f}(\nabla u,\nabla u) + \langle\nabla\Delta_{f}u,\nabla u\rangle.
\]
Combined with the shrinker equation and $-\Delta_{f}u=\alpha u$, we simplify to
\[
\frac12\Delta_{f}|\nabla u|^{2} = |\nabla^{2}u|^{2} + (1-\alpha)|\nabla u|^{2}.
\]
Multiply both sides by $\chi_{R}^{2}$ and integrate by parts:
\[
-\frac12\int_{X}\langle\nabla\chi_{R}^{2},\nabla|\nabla u|^{2}\rangle\,d\mu
= \int_{X}\chi_{R}^{2}\left(|\nabla^{2}u|^{2} + (1-\alpha)|\nabla u|^{2}\right)d\mu.
\]
Denote the left-hand side by $E_{R}$. We estimate
\begin{align*}
|E_{R}|
&\leq \int_{X}2\chi_{R}|\nabla\chi_{R}|\cdot|\nabla^{2}u|\cdot|\nabla u|\,d\mu \\
&\leq \varepsilon\int_{X}\chi_{R}^{2}|\nabla^{2}u|^{2}\,d\mu + \frac{1}{\varepsilon}\int_{X}|\nabla\chi_{R}|^{2}|\nabla u|^{2}\,d\mu.
\end{align*}
If $\alpha \le 1$, then $(1-\alpha)|\nabla u|^{2} \ge 0$, so
\[
(1-\varepsilon)\int_{X}\chi_{R}^{2}|\nabla^{2}u|^{2}\,d\mu
\le \frac{1}{\varepsilon}\int_{X}|\nabla\chi_{R}|^{2}|\nabla u|^{2}\,d\mu.
\]

If $\alpha > 1$, we have
\[
(1-\varepsilon)\int_{X}\chi_{R}^{2}|\nabla^{2}u|^{2}\,d\mu
\le \frac{1}{\varepsilon}\int_{X}|\nabla\chi_{R}|^{2}|\nabla u|^{2}\,d\mu
+ (\alpha-1)\int_{X}|\nabla u|^{2}\,d\mu.
\]

Since $\nabla u\in L^{2}(d\mu)$ and $|\nabla\chi_{R}|\le \dfrac{C}{R}$, we obtain
\[
\int_{B_{R}(p)}|\nabla^{2}u|^{2}\,d\mu
\le \int_{X}\chi_{R}^{2}|\nabla^{2}u|^{2}\,d\mu \le C.
\]

Letting $R\to\infty$ and applying the monotone convergence theorem, we get $\nabla^{2}u\in L^{2}(d\mu)$.

Next we prove $E_{R}\to 0$. By the Cauchy--Schwarz inequality,
\begin{align*}
|E_{R}|
&\le \int_{X}2\chi_{R}|\nabla\chi_{R}|\,|\nabla^{2}u|\,|\nabla u|\,d\mu \\
&\le 2\left( \int_{B_{2R}(p)\setminus B_{R}(p)} |\nabla^{2}u|^{2}\,d\mu \right)^{\frac12}
\left( \int_{X} |\nabla\chi_{R}|^{2}|\nabla u|^{2}\,d\mu \right)^{\frac12}.
\end{align*}

The first term on the right-hand side tends to zero because $\nabla^{2}u\in L^{2}(d\mu)$ and the integral domain is an annulus going to infinity.
The second term also tends to zero, since $|\nabla\chi_{R}|\le \dfrac{C}{R}$ and $\nabla u\in L^{2}(d\mu)$.
Hence $E_{R}\to 0$, and we have the equality
\[
\int_{X}\chi_{R}^{2}\left(|\nabla^{2}u|^{2}+(1-\alpha)|\nabla u|^{2}\right)d\mu = E_{R}.
\]

Letting $R\to\infty$, and using $\nabla u,\nabla^{2}u\in L^{2}(d\mu)$, we arrive at the global integral identity:
\[
\int_{X}|\nabla^{2}u|^{2}\,d\mu = (\alpha-1)\int_{X}|\nabla u|^{2}\,d\mu.
\]

If $u$ is nonconstant, then $\alpha\ge 1$. 

If $\operatorname{Ric}>0$ and $\alpha=1$, we get $\nabla^{2}u\equiv 0$.
From the standard Bochner formula we have $\operatorname{Ric}(\nabla u,\nabla u)= 0$.
This leads to a contradiction with $\operatorname{Ric}>0$, so we must have $\alpha>1$.

If $\alpha=1$, then $\nabla^2 u\equiv 0$. Let $(\widetilde{X},\widetilde{g},\widetilde{f})$ denote the universal cover shrinker satisfying $\operatorname{Ric}_{\tilde{g}} + \nabla^2_{\tilde{g}} \tilde{f} = \tilde{g}$, with $\widetilde{g}=\pi^*g$, $\widetilde{f}=\pi^*f$ the pullbacks under the covering map $\pi$. By the de Rham decomposition theorem \cite[Theorem 10.43]{besse2007einstein}, $\widetilde{X}$ isometrically splits to $\mathbb{R}\times N$. The flat $\mathbb{R}$-factor yields
\[
\tilde{f}(s,y) = \frac{1}{2} s^2 + a s + f_N(y).
\]
Translating the $s$-coordinate removes the linear term $as$, so $\tilde{f}(s,y) = \frac12 s^2 + f_N(y)$. Restricting to the factor $N$, we obtain
\[
\operatorname{Ric}_{g_N} + \nabla^2_{g_N} f_N = g_N.
\]
\end{proof}
At the end of this section, we prove a technical proposition which we shall invoke frequently in all subsequent arguments.
\begin{proposition}\label{3.4}
    Let $(X,g,f)$ be a complete Ricci shrinker,
Let $q\in X$ be a critical point of $f$, and suppose that
$\phi\in C^\infty(X,\mathbb R)$ satisfies $\nabla f(\phi)=\alpha\phi$ for some $\alpha\in\mathbb R$. If $d\phi(q)\neq0$, then
\[
\bigl(\operatorname{Ric}_q^\sharp\bigr)^*d\phi(q)
=(1-\alpha)d\phi(q).
\]
In particular, $1-\alpha$ is an eigenvalue of the Ricci
endomorphism at $q$. Consequently, if $\operatorname{Ric}_q\geq0\, (\operatorname{Ric}_q>0)$, then $\alpha\leq 1 \,(\alpha<1)$.
\end{proposition}

\begin{proof}
We set $Z=\nabla f$, since $q$ is a critical point of $f$, one has $Z(q)=0$. The
linearization
\[
dZ_q:T_qX\longrightarrow T_qX
\]
is therefore canonically defined and is given by $dZ_q(v)=(\nabla_vZ)_q$. 

Differentiating the identity $Z(\phi)=\alpha\phi$ at $q$, we obtain, for every $v\in T_qX$,
\begin{align*}
\alpha\,d\phi_q(v) =d\bigl(Z(\phi)\bigr)_q(v) =\nabla^2\phi_q\bigl(v,Z(q)\bigr)
+d\phi_q\bigl((\nabla_vZ)_q\bigr) =d\phi_q(dZ_q(v)).
\end{align*}
Hence
\[
(dZ_q)^*d\phi(q)=\alpha\,d\phi(q).
\]

Since $Z=\nabla f$, the shrinker equation gives
\[
dZ_q
=\bigl(\nabla^2f\bigr)_q^\sharp
=I-\operatorname{Ric}_q^\sharp.
\]
It follows that
\[
\bigl(\operatorname{Ric}_q^\sharp\bigr)^*d\phi(q)
=(1-\alpha)d\phi(q).
\]

Because $d\phi(q)\neq0$, the number $1-\alpha$ is an eigenvalue
of $\bigl(\operatorname{Ric}_q^\sharp\bigr)^*$. Since
$\operatorname{Ric}_q^\sharp$ is self-adjoint, it has the same
eigenvalues as its dual. Thus $1-\alpha$ is a Ricci eigenvalue
at $q$.

If $\operatorname{Ric}_q\geq0$, then $1-\alpha\geq0$, so
$\alpha\leq1$. If $\operatorname{Ric}_q>0$, then
$1-\alpha>0$, and hence $\alpha<1$.
\end{proof}
\subsection{The equality case and Gaussian splitting}
For convenience of exposition and to characterize the splitting behavior of Kähler Ricci shrinkers, we introduce the following notation.

Let $E_1 := \{\phi \in R_X : \nabla f(\phi) = \phi\}$ denote the weight-$1$ subspace, and set $D(E_1) := \operatorname{Span}_{\mathbb{C}}\{d\phi : \phi \in E_1\}$. 
\begin{align*}
    D\colon E_1 \to H^0(X,T^{1,0*}X), \qquad \phi \mapsto d\phi = \partial\phi
\end{align*}
By definition of $E_1$, $\ker D = 0$, so $E_1 \cong D(E_1)$ as complex vector spaces.

For any \(\phi\in R_X\), the \(T\)-weight decomposition gives \(\phi=\sum_{\beta}\phi_{\beta}\) with \(\phi_{\beta}\in R_{X,\beta}\). Substitute into \(Z(\phi)=\phi\):
\begin{align*}
    \sum_{\beta}\alpha_{\beta}\phi_{\beta} = \sum_{\beta}\phi_{\beta},
\end{align*}
which forces \(\alpha_{\beta}=1\) for all \(\beta\) appearing in the sum. Hence
\begin{align*}
    E_1 = \bigoplus_{\alpha_{\beta}=1} R_{X,\beta}
    = \bigoplus_{\langle\beta,\xi\rangle=1} R_{X,\beta}.
\end{align*}

Combining the simple connectedness of Kähler Ricci shrinkers \cite[Proposition 3.10]{sun2024k}\cite{esparza2025shrinking} and the properties of homogeneous regular functions, we establish the following theorem.

\begin{theorem}\label{3.3}
 Let $(X,g,J,f)$ be a complete K\"ahler-Ricci shrinker.
Then for every nonconstant homogeneous regualr function  $\phi\in R_{X,\beta}$, we have $\alpha_{\beta} \ge 1$. More precisely:
\begin{enumerate}
    \item If $\operatorname{Ric} > 0$ everywhere on $X$, then $\alpha_{\beta} > 1$.
    \item If $\alpha_{\beta} = 1$, then $X$ is holomorphically isometric to $\mathbb{C} \times N$, where $N$ itself is a complete K\"ahler-Ricci shrinker.
    \item If $\dim_{\mathbb{C}}E_1=k$, then $X$ is holomorphically isometric to $\mathbb{C}^k \times N$, where $N$ itself is a complete K\"ahler-Ricci shrinker.
\end{enumerate}
\end{theorem}
\begin{proof}
    \begin{enumerate}
        \item Since $\phi$ is a homogeneous regular function and non-constant, its real and imaginary parts cannot both be constant. Without loss of generality, we suppose $u=\operatorname{Re} \phi$. By Lemma \ref{3.1}, $u\in L^2(e^{-f}dV)$, and Proposition \ref{3.2} directly forces $\alpha_{\beta}\geq 1$, if $\operatorname{Ric}>0$ on $X$, then $\alpha_{\beta}>1$.
        \item If $\alpha_{\beta}=1$, Proposition \ref{3.2} yields $\nabla^2 u \equiv 0$, meaning $\nabla u$ is a parallel vector field. By the K\"ahler condition, $J\nabla u$ is likewise parallel. Combined with the simple connectedness of complete K\"ahler-Ricci shrinkers and the de Rham decomposition theorem\cite[Theorem 10.43]{besse2007einstein}, we conclude that $X$ is holomorphically isometric to $\mathbb{C} \times N$. The $\mathbb{C}$-factor is Ricci-flat, and the shrinker equation degenerates to $\sqrt{-1}\partial\bar\partial f_{\mathbb{C}} = \omega_{\mathbb{C}}$. This gives the splitting $f(z,w) = {|z|^2}/{2} + f_{N}(w)$, which implies that $N$ itself is a complete K\"ahler-Ricci shrinker.
        \item  As established previously, every $\phi_{\beta}$ with $\alpha_{\beta}=1$ admits a parallel vector field $\nabla u_{\beta}$, with $u_{\beta}=\operatorname{Re}\phi_{\beta}$. Due to $\dim_{\mathbb{C}}E_1=k$, we pick \(\{\phi_1,\dots,\phi_k\}\subset E_1\) so that \(\{d\phi_1,\dots,d\phi_k\}\) is a basis of \(D(E_1)\). From the parallelism of the vector field, we obtain that the parallel vector fields \(\{\nabla u_1,\dots,\nabla  u_k\}\) are globally linearly independent, spanning a parallel distribution \(\mathcal{D}\). The Kähler condition ensures \(J\mathcal{D}\) is parallel, Since the functions \(\phi_1,\ldots,\phi_k\) are holomorphic and \(d\phi_1,\ldots,d\phi_k\) are complex-linearly independent, theirreal and imaginary parts have real-linearly independent differentials. Moreover, $\nabla\operatorname{Im}\phi_i=J\nabla\operatorname{Re}\phi_i$. Hence these parallel vector fields span a real \(2k\)-dimensional \(J\)-invariant parallel distribution. so the de Rham decomposition theorem yields a holomorphic isometric decomposition \(X\cong\mathbb{C}^k\times N\), where \(N\) is a complete Kähler Ricci shrinker from the preceding discussion.
        
    \end{enumerate}
   
\end{proof}

\section{The first-order visibility of Kähler Ricci shrinker}
\subsection{Kähler Ricci shrinker surfaces are first order visible}
In this section, we will prove the following theorem and discuss its geometric applications.
\begin{theorem}\label{4.2}
    Any complete noncompact Kähler Ricci shrinker surfaces are first-order visible.
\end{theorem}

We first fix the basic setup. Let $(X,g,J,f)$ be a smooth Kähler Ricci shrinker surface, and let $\pi\colon X\to Y$ denote its associated polarized Fano fibration, where $Y$ is a normal affine variety. Let $T$ be the compact torus defined as the closure of the flow generated by $J\nabla f$, the fibration $\pi$ is $T$-equivariant and there exists unique $T$-fixed point on $Y$.
From an algebro-geometric perspective, $X$ is quasi-projective and $Y$ remains a normal affine variety. Throughout this section, $\mathcal{O}_X$ and $\mathcal{O}_Y$ stand for the algebraic structure sheaves of $X$ and $Y$, respectively. We define the corresponding coordinate rings by
\[
R_X:=\Gamma(X,\mathcal O_X),\qquad R_Y:=\Gamma(Y,\mathcal O_Y).
\]
The canonical isomorphism $\pi_*\mathcal O_X=\mathcal O_Y$ implies that the pullback morphism induces a ring isomorphism
\[
\pi^*\colon R_Y\xrightarrow{\sim} R_X.
\]
Consequently, both rings admit compatible weight decompositions with respect to the $T$-action:
\[
R_X=\bigoplus_{\beta\in \operatorname{Lie}(T)^*}R_{X,\beta},
\qquad
R_Y=\bigoplus_{\beta\in \operatorname{Lie}(T)^*}R_{Y,\beta}.
\]
Under this grading, all constant functions carry zero weight, whereas every nonconstant homogeneous regular function carries positive weight.

\subsubsection{The case of $\operatorname{dim}_{\mathbb{C}}Y = 1$}
\begin{proof}
    Firstly, by proposition \ref{A1}, we have $Y\cong \mathbb{A}^1$.
    Let $t\colon Y\to \mathbb{C}$ be the homogeneous coordinate on $Y$, Since $t$ is a global regular function on $Y$, its pullback $\phi=\pi^*t$ is a global regular function on $X$. Moreover, since $t$ is homogeneous and $\pi$ is $T$-equivariant, $\phi$ is again homogeneous. Thus $\phi$ is a globally defined homogeneous regular function on $X$.

    We first identify the general fiber. By generic smoothness \cite[Corollary 10.7]{hartshorne2013algebraic}, there exists a nonempty Zariski open subset $U \subset Y$ such that $\pi^{-1}(U) \to U$ is smooth. Since \(\pi\) is projective and
\(\pi_*\mathcal O_X=\mathcal O_Y\), its fibers are projective and
connected. Consequently, for every closed point \(y\in U\), the fiber $F_y:=\pi^{-1}(y)$
is a smooth, connected projective curve.
    
    \begin{claim}
         For every closed point \(y\in U\), the fiber \(F_y\) is isomorphic to $\mathbb P^1$.
    \end{claim}
    \begin{proof}
        We set $F=F_y$, since $-K_X$ is $\pi$-ample, we have $-K_X\cdot F>0$. And the self-intersection satisfies $F^2=0$. By the adjunction formula $2g(F)-2 = K_X\cdot F + F^2 $,
    substituting $F^2=0$ rearranges to
    \[
    2-2g(F) = -K_X\cdot F > 0.
    \]
    This implies $g(F)=0$, so $F\cong \mathbb{P}^1$ and $-K_X\cdot F=2$.
    \end{proof}

We next analyze the central fiber 
\begin{align*}
    F_o =: \pi^{-1}(o) = \sum_i m_i C_i,
\end{align*}
where each $C_i$ is a compact reduced and irreducible curve in $X$. Let $d_i := -K_X\cdot C_i > 0$ (this positivity follows from $\pi$-ampleness of $-K_X$).
Since $F_o$ is numerically equivalent to the generic fiber $F$, we have $-K_X\cdot F_o = -K_X\cdot F = 2$, hence
\[
\sum_i m_i d_i = 2.
\]
As $m_i,d_i$ are positive integers, there are only three possible forms for $F_o$:
\begin{enumerate}
    \item[(i)] $F_o = C$, with $-K_X\cdot C = 2$;
    \item[(ii)] $F_o = 2C$, with $-K_X\cdot C = 1$;
    \item[(iii)] $F_o = C_1 + C_2$, with $-K_X\cdot C_1 = -K_X\cdot C_2 = 1$.
\end{enumerate}

We treat each case separately.

\noindent(i): $F_o = C$

Here $F_o =C$ is a fiber with $-K_X\cdot C=2$. Numerical equivalence to a general fiber gives $C^2 = F\cdot C = 0$. Applying the adjunction formula to $C$:
\[
2p_a(C)-2 = K_X\cdot C + C^2,
\]
which simplifies to
\[
2 = 2 - 2p_a(C),
\]
where $p_a(C)$ denotes the arithmetic genus of $C$.  Therefore, we have $p_a(C)=0$, which implies $C\cong\mathbb{P}^1$. Since $\pi$ is $T$-equivariant and $o$ is the fixed point, so $C$ is invariant under the torus action, By proposition \ref{FFF}, we know the vector field $J\nabla f$ vanishes somewhere on $C$, and we set $Z=\nabla f$, namely there exists a point $q\in C$ such that $Z(q)=0$.

The divisor of $\phi$ is precisely the central fiber: $\operatorname{div}_{X}(\phi) = F_o $, by assumption, $F_o = C$ is reduced and the homogeneous coordinate function has vanishing order $1$, so
\begin{align*}
    \operatorname{div}_{X}(\phi)=\pi^*\operatorname{div}_Y(t)=\pi^*(1\cdot[o])=F_o
\end{align*}

Since $C$ is smooth at $q$, there exists a holomorphic coordinate system $(z,w)$ centered at $q$ such that $C = \{z = 0\}$ near $q$. In the local ring $\mathcal{O}_{X,q}$, the ideal of $C$ is therefore generated by $z$: $I_{C,q} = (z)$.
On the other hand, since $\operatorname{div}(\phi) = C$ with multiplicity one, the principal ideal generated by $\phi$ defines the same Cartier divisor: $(\phi) = (z) \subset \mathcal{O}_{X,q}$.
Consequently, $\phi$ and $z$ differ by a unit in $\mathcal{O}_{X,q}$. Thus there exists a holomorphic unit $h \in \mathcal{O}_{X,q}^{\times}$ such that $\phi = h z$,
here $h$ being a unit means $h(q) \neq 0$.

Taking differentials gives
\[
d\phi = h\,dz + z\,dh.
\]
Since $q \in C = \{z = 0\}$, we have $z(q) = 0$. Therefore
\[
d\phi(q) = h(q)\,dz(q).
\]
Because $h(q) \neq 0$ and $z$ is a local coordinate, $dz(q) \neq 0$. Hence
\[
d\phi(q) \neq 0.
\]

\noindent(ii) $F_o=2C$

In this case, $-K_X\cdot C=1$ and $F_o\cdot C=0$, so $2C^2=0$. Applying the adjunction formula:
\[
1=-K_X\cdot C=2-2p_a(C)+C^2=2-2p_a(C).
\]
The right-hand side is an even integer, which leads to a contradiction. Hence this case cannot occur.

\noindent(iii) $F_o=C_1+C_2$

Here we have $-K_X\cdot C_1=-K_X\cdot C_2=1$. Let $a=C_1\cdot C_2$.
From $F_o\cdot C_1=0$, we get $C_1^2+a=0$. Similarly, $C_2^2+a=0$.
Now apply the adjunction formula to $C_1$:
\[
1=-K_X\cdot C_1=2-2p_a(C_1)+C_1^2=2-2p_a(C_1)-a.
\]
It follows that $p_a(C_1)=0$ and $a=1$. By symmetry, we also have $p_a(C_2)=0$.
Thus $C_1\cong C_2\cong\mathbb{P}^1$, with self-intersections $C_1^2=C_2^2=-1$ and intersection number $C_1\cdot C_2=1$.
In other words, $C_1$ and $C_2$ intersect transversely at a single point. 
\begin{claim}
$C_i$ is $T$-invariant.
\end{claim}
\begin{proof}
    Since $\pi$ is $T$-equivariant and $o \in Y$ is fixed by $T$, the central fiber $F_o=\pi^{-1}(o)$ is $T$-invariant. Hence $T$ acts on the finite set of irreducible components $\{C_1,C_2\}$, this gives a group homomorphism $T\longrightarrow \operatorname{Sym}\{C_1,C_2\}\cong S_2$.
Since $T$ is connected and $S_2$ is discrete, the image is a single point. As the identity element of $T$ fixes each $C_i$, the homomorphism is trivial. Therefore each component $C_i$ is individually $T$-invariant.
\end{proof}

\begin{claim}
There exists a point $q\in F^{sm}_o$ such that $Z(q)=0$.
\end{claim}
\begin{proof}
    Let $F_o^{\mathrm{sm}}=(C_1\cup C_2)\setminus (C_1\cap C_2)$ denote the smooth part of the central fiber. It follows that the $T$-action restricts to a holomorphic action on each $C_i\cong \mathbb{P}^1$, By proposition \ref{FFF}, the $T$-action admits at least two fixed points. Therefore we can choose such a fixed point away from the node $C_1\cap C_2$, Hence there exists $q\in F_o^{\mathrm{sm}}$ such that $q$ is fixed by $T$, so we have $Z(q)=0$. 
\end{proof}

It remains to check that $d\phi(q)\neq 0$. Since $q\in F_o^{\mathrm{sm}}$ and $F_o=\operatorname{div}(\phi)$ is reduced at $q$, the divisor $F_o$ is locally a smooth reduced curve near $q$. Thus there exist local holomorphic coordinates $(z,w)$ centered at $q$ such that $F_o=\{z=0\}$ near $q$. In the local ring $\mathcal{O}_{X,q}$, the ideal of $F_0$ is generated by $z$. On the other hand, since $F_o=\operatorname{div}(\phi)$ with multiplicity one near $q$, the principal ideal $(\phi)$ defines the same reduced Cartier divisor. Hence $(\phi)=(z)\subset \mathcal{O}_{X,q}$.
Therefore there exists a unit $h\in \mathcal{O}_{X,q}^{\times}$ such that $\phi=hz$.

Taking differentials gives
\[
d\phi=h\,dz+z\,dh.
\]
Since $z(q)=0$, we get
\[
d\phi(q)=h(q)\,dz(q).
\]
Here $h(q)\neq 0$ because $h$ is a unit, and $dz(q)\neq 0$ because $z$ is a local coordinate. Consequently,
\[
d\phi(q)\neq 0.
\]
\end{proof}

\subsubsection{The case of $\operatorname{dim}_{\mathbb{C}}Y = 2$}

\begin{proof}

In this case, $X$ is smooth and $Y$ is a normal variety, $\pi\colon X\to Y$ is a birational morphism. Indeed, $\pi$ is projective and satisfies $\pi_*\mathcal{O}_X=\mathcal{O}_Y$, which forces all fibers of $\pi$ to be connected \cite[Corollary 11.3]{hartshorne2013algebraic}. Since $\dim X=\dim Y=2$, the general fiber of $\pi$ is a single point. By the standard relation between the degree of a morphism and the cardinality of its general fiber, the field extension degree $[\mathbb{C}(X):\mathbb{C}(Y)]$ equals $1$, so $\mathbb{C}(X)\cong\mathbb{C}(Y)$ and $\pi$ is birational.

Let $o\in Y$ denote the $T$-fixed point. Owing to the $T$-equivariance of $\pi$, we separate our analysis into two distinct cases.

\noindent (i) $\pi^{-1}(o)$ is a single point.

Write $\pi^{-1}(o)=\{q\}$. The equivariance implies that $q$ is a $T$-fixed point, so $Z(q)=0$.
Since $\pi$ is birational, $Y$ is normal, and no positive-dimensional fiber of $\pi$ passes through $q$, the Zariski Main Theorem \cite[Chapter III, §9]{mumford2004red} implies that $\pi$ restricts to an isomorphism on a neighbourhood of $q$. In particular, $Y$ is smooth at $o$.

We now construct a global homogeneous function $\psi\in R_Y$ such that $d\psi(o)\neq 0$. Since $Y$ is affine, the cotangent space satisfies $T^*_oY\cong \mathfrak{m}_o/\mathfrak{m}_o^2$, where $\mathfrak{m}_o$ is the maximal ideal of $R_Y$, more precisely, $\mathfrak{m}_o=\{\psi \in R_Y : \psi(o) = 0\} \subset R_Y $, 
The ring $R_Y$ is positively graded, and $\mathfrak{m}_o$ is a homogeneous maximal ideal. Since $o$ is the $T$-fixed point, $t\cdot \psi(o)=0$ induces that $\mathfrak{m}_{o}$ is $T$-invariant. Combined with the $T$-graded decomposition of $R_Y$, we conclude that $\mathfrak{m}_o$ is a homogeneous maximal ideal, i.e.,
\begin{align*}
    \mathfrak{m}_o = \bigoplus_{\beta} \left(\mathfrak{m}_o \cap R_{Y,\beta}\right),
\end{align*}
hence the quotient $\mathfrak{m}_o/\mathfrak{m}_o^2$ also inherits the grading. so we can choose a homogeneous representative $\psi\in R_Y$ with $[\psi]\neq 0$, which is equivalent to $d\psi(o)\neq 0$.

Pull back this function and set $\phi=\pi^*\psi$. Since $\psi$ is a global homogeneous regular function on $Y$, $\phi$ is regular on $X$.
Moreover, $\phi$ is homogeneous, consistent with the equivariance of $\pi$.
As $\pi$ is locally an isomorphism near $q$, we have $d\phi(q)\neq 0$. 

\noindent (ii) $\operatorname{Supp}\,\pi^{-1}(o) = \bigcup_{i=1}^r C_i$.

The curves $C_1,\dots,C_r$ are the exceptional components of $\pi^{-1}(0)$. These curves are contracted by the birational morphism $\pi$, so their intersection matrix is negative definite.
In particular, $C_i^2<0$. The $\pi$-ampleness of $-K_X$ gives $-K_X\cdot C_i>0$.
We apply the adjunction formula:
\[
0<-K_X\cdot C_i=2-2p_a(C_i)+C_i^2\le 2-2p_a(C_i).
\]
This forces $p_a(C_i)=0$ and $C_i^2=-1$. Therefore, every exceptional component is a $(-1)$-curve.

\begin{claim}
The fiber $\pi^{-1}(o)$ contains exactly one $(-1)$-curve.
\end{claim}
\begin{proof}
    Suppose for contradiction that $r\ge 2$. Let $E = C_1+\dots+C_r$. Since the fiber is connected, the dual graph of the exceptional curves is connected, so
\[
e:=\sum_{i<j} C_i\cdot C_j \ge r-1.
\]
We compute the self-intersection:
\[
E^2 = \sum_{i} C_i^2 + 2\sum_{i<j} C_i\cdot C_j = -r + 2e \ge -r + 2(r-1) = r-2.
\]
If $r\ge 2$, then $E^2\ge 0$. This contradicts the fact that the intersection matrix is negative definite. Therefore we must have $r=1$.
\end{proof}

By Castelnuovo’s contraction theorem \cite[Theorem 1.13.3]{kawamata2024algebraic}, there exists smooth surface $Y'$ and the blow down $\sigma:X\to Y'$ s.t. $\sigma(C)=o'$, combining with $\pi(C)=o$ we have $\pi=h\circ \sigma$ where $h:Y'\to Y$ is the finite birational morphism near $o'$, and $Y$ is normal so $h$ is an isomorphism near $o'$. Consequently, the point $o$ is smooth on $Y$.
The point $o$ is the unique $T$-fixed point, so the curve $C=\pi^{-1}(o)$ is $T$-invariant.
By proposition \ref{FFF}, any holomorphic torus action on $\mathbb{P}^1$ has fixed points, so there exists some $q\in C$ such that $Z(q)=0$.

We now construct a function $\phi$ satisfying $d\phi(q)\neq 0$. Locally near $C$, the morphism $\pi$ is the blow-down of the $(-1)$-curve $C$ to the smooth point $o\in Y$. The point $q\in C\cong \mathbb{P}(T_o Y)$ corresponds to a line $l_q\subset T_o Y$. Since $q$ is a $T$-fixed point, $l_q$ is a $T$-invariant line.
The torus representation on the cotangent space $T_o^*Y\cong \mathfrak{m}_o/\mathfrak{m}_o^2$ is diagonalizable.
We can choose a homogeneous covector $\eta\in T_o^*Y$ such that $\eta(l_q)\neq 0$.
Because $Y$ is affine and its coordinate ring $R_Y$ is graded, we can pick some $\psi\in R_Y$ such that $[\psi]=\eta\in \mathfrak{m}_o/\mathfrak{m}_o^2$.
The element $\psi$ is therefore a globally defined homogeneous regular function on $Y$.
Define $\phi=\pi^*\psi$, then $\phi$ is a global homogeneous regular function on $X$.

We verify that $d\phi(q)\neq 0$, Choose local coordinates $(x,y)$ at $o \in Y$ and the standard blow-up chart $(u,v)$ so that
\[
\pi(u,v) = (u,uv),
\]
and so that the chosen point $q \in C \cong \mathbb{P}^1$ is given by $\{u = 0, v = 0\}$. The corresponding line in $T_oY$ is $\ell_q = \mathbb{C}\partial_x$. Choose a homogeneous covector $\eta \in T_o^*Y$ such that $\eta(\ell_q) \neq 0$. Writing
\[
\psi = ax + by + h(x,y),
\]
where the notation $h(x,y)$ stands for terms vanishing to order at least two at $o$ i.e. $h(x,y)\in (x,y)^2$, so $h(u,uv)\in(u^2)$, and $\eta(\ell_q) \neq 0$ means $a \neq 0$. Hence
\[
\pi^*\psi = au + buv + h(u,uv),
\]
the linear term at $q$ is $au$. Therefore
\[
d(\pi^*\psi)(q) = a\,du \neq 0.
\]
This implies $d(\pi^*\psi)(q)\neq 0$, i.e., $d\phi(q)\neq 0$. 

\end{proof}
%\begin{remark}
 %  If $\dim_{\mathbb{C}}X=\dim_{\mathbb{C}}Y=n$, the same reasoning implies that $\pi$ is a birational morphism.
%Suppose $\dim_{\mathbb{C}}\pi^{-1}(o)=0$; then $\pi^{-1}(o)$ reduces to a single point $\{q\}$.
%By the arguments from Section 4.1.2, there exists a neighbourhood $U$ of $o$ such that $\pi:\pi^{-1}(U)\to U$ is an isomorphism.
%Via its differential $d\pi_q$, this yields an isomorphism $T_q^{1,0}X\cong T_o^{1,0}Y$.
%Taking the pullback by $d\pi_q^*$, we further obtain $\mathfrak{m}_o / \mathfrak{m}_o^2 \cong T_q^{*1,0}X$.
%Hence we may choose $\phi_1,\dots,\phi_n\in\mathfrak{m}_q$ such that $\operatorname{span}_{\mathbb{C}}\{d\phi_1,\dots,d\phi_n\}=T_q^{*1,0}X$.
%Combining this with Proposition \ref{3.4} and Theorem \ref{3.3}, we conclude that $X$ is the Gaussian $\mathbb{C}^n$ shrinker. This observation serves as the starting point for our proof of Theorem \ref{55}.

%Consequently, any non-Gaussian shrinker must satisfy $\dim_{\mathbb{C}}\pi^{-1}(o)\geq 1$.
%In particular, a nontrivial birational-type Kähler Ricci shrinker cannot be locally affine over the vertex of its canonical polarized affine cone.
%Its non-Gaussian geometry necessarily induces genuine exceptional geometry above the cone vertex; e.g. FIK shrinker.
%\end{remark}
\subsubsection{The Geometric Applications on Kähler Ricci shrinker surface}

In this section, we prove the Theorem \ref{main}. For the convenience of the reader, we restate the theorem below.

\begin{theorem}\label{4.4}
Let $(X^2,g,J,f)$ be a complete Kähler Ricci shrinker surface.
\begin{enumerate}
    \item If $\operatorname{Ric}>0$, then $X$ is compact i.e. Fano surface;
    \item If $X$ is noncompact with $\operatorname{Ric}\geqslant 0$, then $X$ is holomorphically isometric to $\mathbb C^2$ or $\mathbb P^1\times\mathbb C$
    with the corresponding Gaussian or product Kähler Ricci shrinker.
\end{enumerate} 
\end{theorem}
\begin{proof}
        1. Since $X$ naturally admits a polarized Fano fibration structure $\pi \colon X \to Y$, we split the subsequent analysis based on the complex dimension $\dim_{\mathbb{C}} Y$ of $Y$.
        
        If $\dim_{\mathbb{C}}Y=0$, then $Y=\operatorname{Spec}\mathbb{C}$ and the fibration reduces to a constant map $\pi\colon X\to \mathrm{pt}$. Since $\pi$ is a polarized Fano fibration, $X$ is projective with $-K_X$ ample. In particular, $X$ is a compact Fano manifold. Thus the desired compactness conclusion is immediate in this case. 
         
        If $\dim_{\mathbb{C}}Y=1,2$, by the Theorem \ref{4.2} we know $X$ is first-order visible, combining with $\operatorname{Ric}>0$ and Proposition \ref{3.4}, we know the corresponding weight $\alpha<1$, whic is contradict to Theorem \ref{3.3}, therefore this case is ruled out.

        2. The Proposition \ref{3.4} implies the weight $\alpha\leqslant 1$, the Theorem \ref{3.3} implies the weight $\alpha\geqslant  1$, so $\alpha=1$. By Theorem \ref{3.3}, we know $X$ is holomorphically isometric to $\mathbb{C}\times N$, wehre $N$ is a $1$ dimensional complete Kähler Ricci shrinker curve, by \cite[Chapter 3]{chow2023ricci}, we know $X$ is holomorphically isometric to $\mathbb{C}^2$ or $\mathbb{P}^1\times \mathbb{C}$ with standard Kähler Ricci shrinker structure.
\end{proof}

\subsection{Toric Kähler Ricci shrinkers are first order visible}

In this section, we will prove the following theorem and discuss its geometric applications.

\begin{theorem}\label{tor}
    Any complete noncompact toric Kähler Ricci shrinkers are first-order visible.
\end{theorem}

We first fix the basic setup. Let $(X,g,J,f)$ be a complete toric Kähler Ricci shrinker. Let $T_{\mathbb{C}}$ denote the algebraic torus $(\mathbb{C}^*)^n$, and let $T^n\subseteq T_{\mathbb{C}}$ be its compact real form. By the definition of a toric manifold and the fan in section 2.2, there is a finite fan
\(\Sigma\) such that \(X\simeq X_\Sigma\), and
\[
    X^{T^n}=X^{T_{\mathbb C}}
    =\{q_\sigma:\sigma\in\Sigma(n)\}.
\]
Hence the torus fixed-point set is finite.

Now we prove the Theorem \ref{tor}.
\begin{proof}
    Since  $H^1(X;\mathbb{R})=0$ \cite{wylie2008complete}, the real torus action is Hamiltonian. Let $\mu:X\longrightarrow\mathfrak t^*$ be a moment map. $J\nabla f\in \mathfrak{t}$ means there exists $\xi\in\mathfrak{t}$ such that fundamental vector field $\xi^{\#}=J\nabla f$, we have
\[
  d\langle\mu,\xi\rangle
  =-\iota_{J\nabla f}\omega
  =d f.
\]
After adding a constant to $f$, we may therefore assume that $f=\langle\mu,\xi\rangle$.

Since $X$ is complete and noncompact, $f$ is proper and bounded from below. Hence the moment-map component
$\langle\mu,\xi\rangle$ is proper and bounded from below. In fact, the moment map $\mu$ is proper.  Indeed, if
$K\subset\mathfrak t^*$ is compact, then $\mu^{-1}(K) \subset f^{-1}\bigl(\langle K,\xi\rangle\bigr)$. The set $\langle K,\xi\rangle\subset\mathbb R$ is compact, so $\mu^{-1}(K)$ is compact.

Let
\begin{align*}
    N_{\ZZ}:=\Hom(\mathbb{C}^*, T_{\mathbb{C}}),\qquad
  M_{\ZZ}:=\Hom(
    T_{\CC},\CC^*)\simeq N_{\ZZ}^\vee.
\end{align*}
be the one-parameter-subgroup and character lattices. We may now apply
the noncompact Delzant theorem in the precise form
\cite[Lemma~2.13]{cifarelli2022uniqueness}.  The hypotheses hold because the fixed-point
set is finite and $\langle\mu,\xi\rangle$ is proper and bounded below. Thus
\begin{align*}
    P:=\mu(X)\subset M_{\RR}:=M_{\ZZ}\otimes_{\ZZ}\RR
\end{align*}
is a full-dimensional Delzant polyhedron, and $(X,\omega)$ is $\TT^n$-equivariantly symplectomorphic to
the symplectic toric manifold associated with $P$.  Because the given
action already extends holomorphically to $\TT^n_{\CC}$,
\cite[Proposition 2.8, Lemma~2.14]{cifarelli2022uniqueness} gives a $\TT^n_{\CC}$-equivariant
biholomorphism $(X,J)\simeq X_{\Sigma_P}$, where $\Sigma_P$ is the normal fan of $P$.

The properness of $\mu$ and the noncompactness of $X$ imply that $P$ is noncompact. Since $P$ is a finite, strongly convex polyhedron with a vertex, it has an unbounded one-dimensional face. 
Choose such a face and write it as
\begin{align*}
    E=v+\RR_{\geq0}m,
\end{align*}
where $v$ is a vertex and $m$ is the primitive lattice generator of the ray determined by $E$ in the character lattice $M_{\ZZ}$. By construction, $m$ is primitive integral and nonzero, so 
$0\neq m\in M_{\mathbb Z}$. Write
\begin{align*}
  P=
    \left\{
        x\in M_{\mathbb R}:
        \langle x,u_j\rangle\ge \lambda_j,\quad
        j=1,\ldots,d
    \right\},
\end{align*}
where \(u_j\in N_{\mathbb Z}\) are the primitive inward-pointing
facet normals. Since $v+tm\in E\subset P$ for all $t\geq 0$, we have $\langle v,u_j\rangle +t\langle m,u_j\rangle \ge\lambda_j$
for all $t\ge0$. Hence $\langle m,u_j\rangle\ge0$ for every $j$.
Therefore, for every $x\in P$ and $t\ge0$,
\begin{align*}
   \langle x+tm,u_j\rangle
    =
    \langle x,u_j\rangle
    +t\langle m,u_j\rangle
    \ge\lambda_j,
\end{align*}
and thus \(x+tm\in P\). It follows that $0\neq m\in\operatorname{rec}(P)\cap M_{\mathbb Z}$.

Consider the toric character $\varphi:=\chi^m$.
By \cite[Theorem 7.1.6]{cox2011graduate}, the support of the normal fan satisfies
\[
  |\Sigma_P|=\rec(P)^\vee.
\]
Consequently, for every cone $\sigma'\in\Sigma_P$ and every
$u\in\sigma'$, one has $\langle m,u\rangle\geq0$.  Equivalently,
$m\in(\sigma')^\vee$, so $\chi^m$ is regular on every affine toric chart
\[
  U_{\sigma'}=\Spec\CC[(\sigma')^\vee\cap M_{\ZZ}].
\]
These local extensions agree on overlaps, since they all restrict to
the same character \(\chi^m\) on the common dense torus
\(T_{\mathbb C}\).  Hence they glue to a global algebraic regular
function
\begin{align*}
    \chi^m\in
    \Gamma(X_{\Sigma_P},\mathcal O_{X_{\Sigma_P}}^{\mathrm{alg}})
\end{align*}
hence a global holomorphic function on $X$.

Let $\sigma_v\in\Sigma_P$ be the maximal cone corresponding to the
vertex $v$, and let $q=q_{\sigma_v}$ be the corresponding torus fixed
point.  The Delzant condition says that the primitive directions
$m_1,\ldots,m_n$ of the edges emanating from $v$ form a $\ZZ$-basis of
$M_{\ZZ}$.  Relabeling if necessary, we may take $m_n=m$.  On the smooth affine
chart $U_{\sigma_v}\simeq\CC^n$, the characters
\[
  z_i:=\chi^{m_i},\qquad i=1,\ldots,n,
\]
are coordinates centered at $q$.  Therefore
\[
  \varphi=z_n,\qquad \varphi(q)=0,\qquad
  d\varphi(q)=d z_n\neq0.                              
\]

It remains to compare the toric algebraic structure with the canonical Sun--Zhang algebraic structure.  Let
\[
    \pi:X\longrightarrow Y=\Spec R_X
\]
be the canonical polarized Fano fibration. We consider its analytification
\[
    \pi^{an}:X^{\mathrm{an}}\longrightarrow Y^{\mathrm{an}}.
\]
 Since \(\pi\) is projective, \(\pi^{\mathrm{an}}\) is proper.
Moreover, its fibers are connected and \(Y^{\mathrm{an}}\) is normal.
Hence analytic Stein factorization \cite{barth2003compact} gives $\pi^{\mathrm{an}}_* \mathcal O_{X^{\mathrm{an}}}^{\mathrm{hol}}= \mathcal O_{Y^{\mathrm{an}}}^{\mathrm{hol}}$. Consequently, every global holomorphic function on \(X^{\mathrm{an}}\)
descends uniquely to \(Y^{\mathrm{an}}\). In particular, there exists a
unique $\psi\in
    \Gamma\!\left(
        Y^{\mathrm{an}},\mathcal O_{Y^{\mathrm{an}}}^\mathrm{hol}
    \right)$ such that $\varphi=(\pi^{\mathrm{an}})^*\psi$.

Let $T_{\mathrm{sol}}:=\overline{\{\exp(t\xi):t\in\mathbb R\}}\subset T^n$. Since \(\varphi=\chi^m\) is a character of the ambient toric torus, its
restriction to \(T_{\mathrm{sol}}\) is a single weight.  Thus there is a
character \(\chi_\beta:T_{\mathrm{sol}}\to S^1\) such that
\[
    a^*\varphi=\chi_\beta(a)\varphi,
    \qquad a\in T_{\mathrm{sol}}.
\]
Since \(\pi\) is \(T_{\mathrm{sol}}\)-equivariant, we have
\[
\begin{aligned}
    \pi^*(a^*\psi)=a^*(\pi^*\psi) =a^*\varphi =\chi_\beta(a)\varphi =\pi^*\bigl(\chi_\beta(a)\psi\bigr).
\end{aligned}
\]
The surjectivity of \(\pi\) implies that \(\pi^*\) is injective, and
therefore $a^*\psi=\chi_\beta(a)\psi$. Hence the \(T_{\mathrm{sol}}\)-orbit of \(\psi\) spans the
one-dimensional space \(\mathbb C\psi\), so \(\psi\) is
\(T_{\mathrm{sol}}\)-finite. By the intrinsic characterization of the affine algebraic structure \cite[\S2.4]{ZhangBirationalKRFNotes}, we have
\begin{align*}
    \Gamma(Y,\mathcal O_Y^{\mathrm{alg}})=\Gamma(Y^{\mathrm{an}},\mathcal O_{Y^{\mathrm{an}}}^{\mathrm{hol}})_{T_{\mathrm{sol}}\text{-finite}},
\end{align*}
and consequently $\psi\in\Gamma(Y,\mathcal O_Y^{\mathrm{alg}})$. It follows that $\varphi=\pi^*\psi\in R_X$.
\end{proof}

We then obtain the following results analogous to the case of Kähler Ricci shrinker surfaces.
\begin{corollary}
    Let $(X^n,g,J,f)$ be a complete toric Kähler Ricci shrinker.
\begin{enumerate}
    \item If $\operatorname{Ric}>0$, then $X$ is compact i.e. toric Fano manifold;
    \item If $X$ is noncompact with $\operatorname{Ric}\geqslant 0$, then $X$ is $T_{\mathbb{C}}$-equivalently holomorphically isometric to $\mathbb{C}\times N$, where $N$ itself is a complete toric Kähler Ricci shrinker.
\end{enumerate} 
\end{corollary}
\begin{proof}
The first assertion follows from Theorem~\ref{4.4}. Suppose that \(X\) is noncompact and \(\Ric\geq0\). Let $\varphi_0=\chi^m$ be the homogeneous regular function constructed above. The Theorem~\ref{4.4} gives a biholomorphic isometric splitting 
\begin{align*}
    \Phi:X\longrightarrow\mathbb C\times N
\end{align*}
such that, after translating the origin of the \(\mathbb C\)-factor
and replacing \(\varphi_0\) by a nonzero constant multiple, so we have $\operatorname{pr}_{\mathbb C}\circ\Phi=\varphi$, where \(\varphi\) has the same \(T_{\mathbb C}\)-weight as
\(\chi^m\). Hence $\varphi(a\cdot x)=\chi^m(a)\varphi(x)$ for $a\in T_{\mathbb C}$. And we have 
\begin{align}\label{789}
    A_a:=\Phi\circ a\circ\Phi^{-1}, \qquad \operatorname{pr}_{\mathbb C}\circ A_a=\chi^m(a)\operatorname{pr}_{\mathbb C}.    
\end{align}
 Set
\[
    \mathcal D_N:=\ker d\operatorname{pr}_{\mathbb C},
    \qquad
    \mathcal D_{\mathbb C}:=\mathcal D_N^\perp .
\]
For \(a\in T\), the identity \ref{789} shows that \(A_a\) preserves \(\mathcal D_N\). Since \(A_a\) is an
isometry, it also preserves \(\mathcal D_{\mathbb C}\). Hence the two
product directions are not mixed, and
\[
    A_a(z,y)=\bigl(\chi^m(a)z,\rho_a(y)\bigr)
\]
for some holomorphic isometry \(\rho_a\) of \(N\).

Both distributions are \(J\)-invariant. Since the infinitesimal action
of \(i\eta\in\mathfrak t_{\mathbb C}\) is \(J\) times that of
\(\eta\in\mathfrak t\), they are preserved by the complexified
infinitesimal action as well. As \(T_{\mathbb C}\) is connected, the
same product form holds for every \(a\in T_{\mathbb C}\). Thus 
\begin{align*}
  A_a(z,y)=\bigl(\chi^m(a)z,\rho_a(y)\bigr), \qquad  a\in T_{\mathbb C}, 
\end{align*}
 where \(\rho:T_{\mathbb C}\curvearrowright N\) is a holomorphic action.
Therefore \(\Phi\) is \(T_{\mathbb C}\)-equivariant.

Let $K_{\mathbb C}:=\ker\chi^m$. Since \(m\) is primitive, so $K_{\mathbb C}\simeq(\mathbb C^*)^{n-1}$. For \(k\in K_{\mathbb C}\), we know $k\cdot(z,y)=(z,\rho_k(y))$. This action is effective: if
\(\rho_k=\Id_N\), then \(k\) acts trivially on
\(\mathbb C\times N\simeq X\), and hence \(k=e\). By the argument of \cite[Section 2.2]{cifarelli2022uniqueness}, we know $N$ is complex toric manifold. Finally, the soliton splitting gives $g=g_{\mathbb C}+g_N$ and $ f=f_{\mathbb C}+f_N$. The induced soliton vector field on \(N\)
belongs to the torus action generated by the compact part of
\(K_{\mathbb C}\). Hence
\((N,g_N,J_N,f_N)\) is a complete toric K\"ahler Ricci shrinker.
\end{proof}
The above results can induce the following corollary.
\begin{corollary}
Let \((X^n,g,J,f)\) be a complete toric K\"ahler Ricci shrinker
with constant scalar curvature. If $\operatorname{Ric}\geq 0$,
then \(X\) is rigid.
\end{corollary}
\begin{proof}
    If $X$ is compact, the proof is straightforward. If \(X\) is noncompact, the toric splitting theorem yields $X\cong\mathbb C\times X_1$,
where \(X_1\) is again a complete toric K\"ahler Ricci shrinker with
nonnegative Ricci curvature. Since the splitting is isometric and
the scalar curvature of \(X\) is constant, the scalar curvature of
\(X_1\) is also constant. Iterating the argument gives $X\cong\mathbb C^k\times N$, where \(N\) is compact. The compact shrinker \(N\) has constant scalar
curvature, and hence its potential is constant by the traced shrinker
equation. Thus \(N\) is K\"ahler--Einstein, proving rigidity.
\end{proof}

\section{Mixed Ricci Signature on Shrinker Examples}

We establish a first-order visibility criterion for Ricci negativity on complete Kähler Ricci shrinkers. The BCCD and some birational contraction-type examples are two manifestations of the same mechanism: their associated Fano fibrations have homogeneous regular functions with nonzero first jet at a soliton fixed point.

\begin{proposition}\label{5.1}
    Let $(X,g,J,f)$ be a complete Kähler Ricci shrinker satisfying the first order visibility. Then  there exists $v \in T_q^{1,0}X$
such that $\operatorname{Ric}_q(v,\bar v) \le 0$.
If moreover $X$ does not split holomorphically and isometrically  as $\mathbb C \times N$,
then the inequality is strict: $\operatorname{Ric}_q(v,\bar v) < 0$.
\end{proposition}
\begin{proof}
    Let $Z=\nabla f$, since $\phi$ is a homogeneous regular function, the same argument from Section 3 yields $Z(\phi) = \alpha \phi$ and $-\Delta_f \phi = \alpha \phi$ for some positive weight $\alpha$. By Proposition \ref{3.2}, we have $\alpha \ge 1$. As $Z(q) = 0$ and $d\phi(q) \neq 0$, the argument used in Proposition \ref{3.4} gives
\[
\bigl(\operatorname{Ric}_q^\sharp\bigr)^*d\phi(q)
=(1-\alpha)d\phi(q).
\]
Let $\{\rho_i\}$ denote the eigenvalues of $\operatorname{Ric}_q^\sharp$. There exists some index $i$ such that $\alpha = 1 - \rho_i$, which immediately implies $\rho_i \le 0$. If $\rho_i = 0$, then $\alpha = 1$, and Theorem \ref{3.3} forces the universal cover of $X$ to split holomorphically and isometrically as $X \simeq \mathbb{C} \times N$, where $N$ is another complete Kähler Ricci shrinker. If such a splitting does not occur for $X$, then we obtain the strict inequality $\operatorname{Ric}_q(v,\bar v) < 0$ along the eigenvector $v$ corresponding to $\rho_i$.
\end{proof}
To rule out the splitting case, we introduce the following result.
\begin{lemma}\label{ffff}
    Let $X$ and $N$ be complex manifolds. Assume $X$ is biholomorphic to $\mathbb{C} \times N$. If $E \subset X$ is a compact complex submanifold, then the projection $E \to \mathbb{C}$ is constant. As a consequence, the holomorphic normal bundle $N_{E/X}$ admits a trivial line subbundle.
\end{lemma}
\begin{proof}
   The maximum principle forces the projection $p\colon E\to\mathbb{C}$ to be constant with value $z_0$. Then $E\subseteq \{z_0\}\times N$ and $TE\subseteq TN\big|_E$, so
\[
N_{E/X} \cong T\mathbb{C}\big|_E \oplus N_{E/(\{z_0\}\times N)}.
\]
This shows that $N_{E/X}$ admits a trivial holomorphic line subbundle.
\end{proof}

\subsection{Application to the BCCD Shrinker}
For BCCD shrinker $\operatorname{Bl}_{p}(\mathbb{P}^1\times \mathbb{C})$, the Fano fibration $\pi:\operatorname{Bl}_{p}(\mathbb{P}^1\times \mathbb{C}) \to \mathbb{C} $ is the composition of the blow-up at $p$ with the projection. We set
\begin{align*}
    \sigma:\operatorname{Bl}_{p}(\mathbb{P}^1\times \mathbb{C})\to \mathbb{P}^1\times \mathbb{C}, \qquad  F=\mathbb{P}^1\times \{0\}\subseteq \mathbb{P}^1\times \mathbb{C}
\end{align*}
$F$ is a smooth curve and $p\in F$ is a smooth point, so we have
\begin{align*}
    \sigma^*F=C+E
\end{align*} 
where the $C$ is the proper transform of $F$ and $E$ is the exceptional curve, so
\begin{align*}
    C^2=(\sigma^*F-E)^2=(\sigma^*F)^2-2\sigma^*F\cdot E+E^2=0-0-1=-1
\end{align*}
This corresponds to case (iii) in Section 4.1.1, by the same argument, there exists two points $q_C,q_E\in(C\cup E)\setminus (C\cap E)$ such that $\operatorname{Bl}_{p}(\mathbb{P}^1\times \mathbb{C})$ satisfies the first order visibility at points $q_C$ and $q_E$. Now we prove the Corollary \ref{BCCD}.
\begin{proof}
     We set $X=\operatorname{Bl}_{p}(\mathbb{P}^1\times \mathbb{C})$. From the above argument, the points $q_C$ and $q_E$ are the first order visible. By the proposition \ref{5.1} and the non-splitting property of $\operatorname{Bl}_{p}(\mathbb{P}^1\times \mathbb{C})$, we know the Ricci curvature admits a negative tangent direction on $q_C$ and $q_E$, because $X$ has positive scalar curvature, the another complex Ricci eigenvalue is positive at the points $q_C$ and $q_E$, so $X$ has mixed Ricci curvature.
\end{proof}

\subsection{Application to the Birational contraction type shrinker}
Let $X = \operatorname{Tot}_{\mathbb P^{n-1}} \mathcal O_{\mathbb P^{n-1}}(-k)$, for $1 \le k < n$, equipped with a complete Kähler Ricci shrinker structure. The zero section will be denoted by  $E \simeq \mathbb P^{n-1}$. $Y=\mathbb C^n / \mu_k$, there exists natural divisor contraction:
\begin{align*}
    \pi \colon \operatorname{Tot}_{\mathbb P^{n-1}} \mathcal O_{\mathbb P^{n-1}}(-k) \longrightarrow \mathbb C^n / \mu_k
\end{align*}
where $\mu_k$ (the group of $k$-th roots of unity) acts diagonally on $\mathbb C^n$. The exceptional locus of $\pi$ is precisely $E$, and $\pi(E) = o$, where $o \in Y$ is the vertex.
\[
R_X = \Gamma(X,\mathcal O_X) \simeq \bigoplus_{m \ge 0} H^0\bigl(\mathbb P^{n-1},\mathcal O_{\mathbb P^{n-1}}(mk)\bigr).
\]
The compact torus action exists fixed points on $\mathbb{P}^{n-1}$, we set a fixed point $q$. Since $\mathcal O_{\mathbb P^{n-1}}(k)$ is globally generated, there exists a section
\[
P \in H^0\bigl(\mathbb P^{n-1},\mathcal O_{\mathbb P^{n-1}}(k)\bigr)
\]
such that $P(q) \neq 0$. Let $\phi_P \in R_X$ be the corresponding regular function on the total space of $\mathcal O_{\mathbb P^{n-1}}(-k)$(by isomorphic). In fact, we have $d\phi_P(q) \neq 0$.

Indeed, choose a local affine chart $U = \{Z_1 \neq 0\} \subset \mathbb P^{n-1}$ with coordinates $w_i = {Z_i}/{Z_1}$ for $i = 2,\dots,n$. And choose a local fiber coordinate $s$ for $\mathcal O_{\mathbb P^{n-1}}(-k)$. In this trivialization, the section $P \in H^0(\mathbb P^{n-1},\mathcal O_{P^{n-1}}(k))$ is represented by a holomorphic function $P(1,w)$, and the corresponding regular function on $X$ is $\phi_P = s\, P(1,w)$.
we note that $\phi_{P}$ is globally defined. This is because $s_\alpha = g_{\alpha\beta}^{-1} \cdot s_\beta$ and $P_\alpha = g_{\alpha\beta} \cdot P_\beta$, so 
\begin{align*}
    s_\alpha \cdot P_\alpha = (g_{\alpha\beta}^{-1} s_\beta) \cdot \left(g_{\alpha\beta} P_\beta\right) = s_\beta \cdot P_\beta.
\end{align*}
The zero section is given by $E = \{s = 0\}$.
At a point $q = (s=0,w=w_0) \in E$, we have $d\phi_P(q) = P(1,w_0)\, ds$. Since $P(q) = P(1,w_0) \neq 0$, it follows that $d\phi_P(q) \neq 0$.  So $X$ satisfies the first-order visibility and $N_{E/X} = \mathcal{O}_{\mathbb{P}^{n-1}}(-k)$. By Proposition \ref{5.1} and Lemma \ref{ffff}, we obtain the following proposition.

\begin{proposition}
    Let $X = \operatorname{Tot}\bigl(\mathcal{O}_{\mathbb{P}^{n-1}}(-k)\bigr)$ with $1 \leq k < n$, equipped with Kähler Ricci shrinker structure, then $X$ cannot satisfy $\operatorname{Ric} > 0$ everywhere. More precisely, there exists some point $p\in X$ and some $v\in T_p^{1,0}X$ s.t $\operatorname{Ric}_p(v,\bar{v})<0$.
\end{proposition}

Chi Li constructed another analogous example \cite{li2010rotationally}. Let $X = \operatorname{Tot}\big(\mathcal{O}_{\mathbb{P}^{n-1}}(-1)^{\oplus k}\big)$ with $k<n$, equipped with a complete Kähler Ricci shrinker metric. Let $E \cong \mathbb{P}^{n-1}$ denote its zero section, and let $Y= C\bigl(\mathbb{P}^{n-1} \times \mathbb{P}^{k-1}\bigr)$ stand for the Segre cone. There exists a small contraction
\begin{align*}
    \pi:\operatorname{Tot}\big(\mathcal{O}_{\mathbb{P}^{n-1}}(-1)^{\oplus k}\big)\longrightarrow C\bigl(\mathbb{P}^{n-1} \times \mathbb{P}^{k-1}\bigr).
\end{align*}
The morphism $\pi$ contracts the zero section $E$ to the vertex of the Segre cone.
\begin{align*}
    R_X = \Gamma(X,\mathcal O_X) \cong\bigoplus_{d\ge 0} H^0\left(\mathbb{P}^{n-1},\mathcal{O}(d)\right)\otimes\operatorname{Sym}^d\left(\mathbb{C}^k\right)^*.
\end{align*}
The zero section $\mathbb{P}^{n-1}$ ensures the existence of a fixed point $q$ for the torus action, and $\mathcal{O}_{\mathbb{P}^{m-1}}(d) \otimes \operatorname{Sym}^d\bigl(\mathbb{C}^k\bigr)^*$ is globally generated, which implies there exist global regular functions on $X$. More precisely, we define
\begin{align*}
    \phi_{s,\ell}(b, v) = s_b\bigl(\ell(v)\bigr),\qquad b\in X, \quad v\in \big(\mathcal{O}_{\mathbb{P}^{n-1}}(-1)^{\oplus k}\big)_b
\end{align*}
where $s \in H^0\left(\mathbb{P}^{n-1}, \mathcal{O}(1)\right)$, $\ell \in \bigl(\mathbb{C}^{k}\bigr)^{*}$, and $\ell(v) = \sum_{a=1}^k \ell_a v_a \in \mathcal{O}(-1)_b$.
By the same reasoning as for the FIK shrinker, $\phi_{s,\ell}$ is a globally regular function satisfying $d\phi_{s,\ell}(q,0)\neq0$. So $X$ satisfies the first-order visibility and $N_{E/X} = \mathcal{O}(-1)^{\oplus k}$. By Proposition \ref{5.1} and Lemma \ref{ffff}, we obtain the following proposition.
 \begin{proposition}
    Let $X = \operatorname{Tot}\big(\mathcal{O}_{\mathbb{P}^{n-1}}(-1)^{\oplus k}\big)$ with $k<n$, equipped with Kähler Ricci shrinker structure, then $X$ cannot satisfy $\operatorname{Ric} > 0$ everywhere. More precisely, there exists some point $p\in X$ and some $v\in T_p^{1,0}X$ s.t $\operatorname{Ric}_p(v,\bar{v})<0$.
\end{proposition}

Our arguments also hold for the Futaki–Wang soliton \cite{futaki2011constructing}. Let $X = \operatorname{Tot}(L^{-k}) \to M$, $K_M = L^{-p}$, $0 < k < p$, where $M$ is a toric Fano Kähler–Einstein manifold and $L$ is a toric ample line bundle. There exists a natural divisorial contraction:
\begin{align*}
    \pi:\operatorname{Tot}(L^{-k}) \longrightarrow Y=\operatorname{Spec}\left(\bigoplus_{d\ge 0} H^0\bigl(M, L^{kd}\bigr)\right),
\end{align*}
where $\pi$ contracts the zero section $M$ to the vertex of the affine cone $Y$.
\begin{align*}
    R_X = \Gamma(X,\mathcal O_X) \cong \bigoplus_{d\ge 0} H^0\bigl(M, L^{kd}\bigr)
\end{align*}
By the Futaki–Wang construction, $J\nabla f = cJr\partial_r$, which generates an $S^1$-action fixing the zero section. Moreover, by \cite[Theorem 6.1.15]{cox2011graduate}, any positive toric line bundle $L$ is very ample on complete toric variety and hence globally generated. We may choose a global section $s\in H^0\bigl(M, L^k\bigr)$ such that $s(q)\neq 0$, and define
\begin{align*}
    \phi_s(x,v)=\left\langle s(x),v\right\rangle, \qquad x\in M, \quad v\in L^{-k}_x.
\end{align*}
In local coordinates, write $v=we^*$ and $s=a(z)e$, where $z$ denotes base coordinates, $w$ the fiber coordinate, while $e$ and $e^*$ stand for a local frame of $L^k$ and its dual frame, respectively. Combined with the condition $s(q)\neq 0$, this implies $d\phi_s(q,0)=a(q)dw\neq 0$. So $X$ satisfies the first-order visibility and $N_{M/X} = L^{-k}$. By Proposition \ref{5.1} and Lemma \ref{ffff}, we obtain the following proposition.
\begin{proposition}
    Let $X = \operatorname{Tot}(L^{-k})$ with $0<k<p$, equipped with Futaki-Wang  Kähler Ricci shrinker structure, then $X$ cannot satisfy $\operatorname{Ric} > 0$ everywhere. More precisely, there exists some point $p\in X$ and some $v\in T_p^{1,0}X$ s.t $\operatorname{Ric}_p(v,\bar{v})<0$.
\end{proposition}

\begin{remark}
    In \cite{cifarelli2024explicit}, Cifarelli constructs a more general family of examples on the total space of a direct sum of line bundles over a Fano Kähler–Einstein manifold $B$. By our preceding arguments, the first-order visibility property depends on the global generation of the line bundle. Consequently, if we assume that $H^0(B,L^{-m_j})\neq 0$ holds for at least one index $j\in\{1,2\}$(Fixed points exist because the torus action restricts solely to the fibers.), the foregoing arguments also hold in Cifarelli's examples.
\end{remark}

\section{Stein Rigidity and Smooth Fano Cone}
In this section, we will prove Theorem \ref{SZCX}. Sun and Zhang conjectured that every Kähler Ricci shrinker on a Fano cone is a Ricci-flat Kähler cone metric; see \cite[Conjecture~6.3]{sun2024k}. Within the polarized Fano fibration framework, this is the case in which the canonical fibration is the identity morphism. Motivated by this conjecture, we consider complete smooth Kähler Ricci shrinkers for which the canonical polarized Fano fibration
\begin{align*}
   \pi:X \longrightarrow Y=\operatorname{Spec}R_X 
\end{align*}
is an isomorphism. In this nonsingular setting, we obtain a stronger conclusion: the shrinker is necessarily the Gaussian shrinker on $\mathbb C^n$. Consequently, the result verifies the Conjecture for Fano cones whose underlying affine variety is nonsingular, including at the cone vertex.
\begin{theorem}[Affine Rigidity]\label{55}
Let $(X^n,g,J,f)$ be a complete Kähler Ricci shrinker, and let
\[
\pi:X\longrightarrow Y=\operatorname{Spec}R_X
\]
be its canonical polarized Fano fibration. If $\pi$ is an
isomorphism, then $(X,g,J,f)$ is holomorphically isometric to the Gaussian shrinker on $\mathbb C^n$.
\end{theorem}
\begin{proof}
    Let $X$ be a Kähler Ricci shrinker. There exists a canonical polarized Fano fibration
\[
\pi \colon X \to Y = \operatorname{Spec} R_X.
\]
By our hypothesis, $\pi$ is an isomorphism, so $X = Y$ is a polarized affine cone. In particular, $X$ admits a torus-fixed point $q \in X$.

Denote by $\mathfrak{m}_q \subseteq R_X$ the maximal ideal of $q$, i.e., $\mathfrak{m}_q = \bigl\{\phi \in R_X \,\big|\, \phi(q) = 0\bigr\}$.
Since $X$ is smooth affine, we have the canonical isomorphism $\mathfrak{m}_q / \mathfrak{m}_q^2 \cong T_q^{*1,0}X$,
which implies $\dim_{\mathbb{C}} \mathfrak{m}_q / \mathfrak{m}_q^2 = n$. The torus action induces a weight decomposition on $\mathfrak{m}_q$, $\mathfrak{m}_q^2$, and their quotient $\mathfrak{m}_q/\mathfrak{m}_q^2$.
By surjectivity of the quotient map and homogeneity of the ideals, we may pick homogeneous functions $\phi_1,\dots,\phi_n \in \mathfrak{m}_q$ such that $\{d\phi_1(q),\dots,d\phi_n(q)\}$ forms a basis of $T_q^{*1,0}X$.
By Theorem \ref{3.3}, these functions satisfy
\[
\nabla f(\phi_i) = \lambda_i \phi_i, \quad \lambda_i \ge 1.
\]

On the other hand, at the torus-fixed point $q$, by the arguments in Proposition \ref{3.4}, we have 
\[
\bigl(\operatorname{Ric}_q^\sharp\bigr)^*d\phi_i(q)
=(1-\lambda_i)d\phi_i(q).
\]
Since $\{d\phi_1(q),\dots,d\phi_n(q)\}$ spans the full cotangent space $T_q^{*1,0}X$, the eigenvalues of $\operatorname{Ric}_q^\#$ are exactly
\[
\rho_i = 1 - \lambda_i \le 0.
\]
The scalar curvature at $q$ is then
\[
R(q) = 2\sum_{i=1}^n \rho_i \le 0.
\]
But scalar curvature $R \ge 0$ everywhere on shrinker $X$, so we must have $R(q) = 0$. This forces $\lambda_i = 1$ for all $i = 1,\dots,n$.
By our Theorem \ref{3.3}, we conclude $X \cong \mathbb{C}^n$, equipped with the standard Gaussian shrinker metric. Alternatively, the strong maximum principle can be invoked.
\end{proof}

\begin{theorem}[Stein Rigidity]
    Let $(X^n,g,J,f)$ be a complete Kähler Ricci shrinker, If the underlying complex manifold X is Stein, then $(X,g,J,f)$ is holomorphically isometric to the Gaussian shrinker on $\mathbb C^n$.
\end{theorem}
\begin{proof}
    Since $X$ is a Stein manifold, we consider the fiber $X_{y}=\pi^{-1}(y)$, due to $\pi$ is a projective morphism, naturally proper morphism, so $X_{y}$ is compact, but Stein manifold don't admit any positive dimensional compact analytic subset, which forces $\operatorname{dim}_{\mathbb{C}}X_{y}=0$, so $\pi$ is an quasi-finite morphism, combining the proper property with $\pi$, we know $\pi$ is a finite morphism.

    Since $\pi$ is finite, it is a naturally affine morphism. Hence, by \cite[Lemma 29.11.3, Tag 01S5]{stacks-project}, combining with $\pi_*\mathcal{O}_X=\mathcal{O}_Y$, we have 
    \begin{align*}
        X \simeq \underline{\operatorname{Spec}}_Y(\pi_*\mathcal O_X) \simeq \operatorname{Spec}_Y(\mathcal{O}_Y)\simeq Y,
    \end{align*}
    so $\pi:X\to Y$ is isomorphic in the sense of algebraic variety. By Theorem \ref{55}, we know $(X,g,J,f)$ is holomorphically isometric to the Gaussian shrinker on $\mathbb C^n$.
\end{proof}

\bibliographystyle{alpha}
\bibliography{ref.bib}

\section*{Author Information}
\noindent\textbf{Tongxin Xu}$^*$\\
School of Mathematical Sciences, Capital Normal University\\
Email: 2250501032@cnu.edu.cn
\vspace{1em}

\noindent\textbf{Zhenlei Zhang}$^\dagger$\\
School of Mathematical Sciences, Capital Normal University\\
Email: zhleigo@aliyun.com\\
\end{document}